\documentclass[12pt]{article}
\usepackage{amsmath,amssymb,graphicx}
\usepackage{color}
\numberwithin{equation}{section}
\numberwithin{figure}{section}
\numberwithin{table}{section}

\newtheorem{theorem}{Theorem}[section]
\newtheorem{corollary}{Corollary}[section]
\newtheorem{proposition}{Proposition}[section]
\newtheorem{statement}{Statement}[section]

\newtheorem{example}{Example}[section]
\newtheorem{definition}{Definition}[section]
\newtheorem{note}{Note}[section]
\newtheorem{problem}{Problem}[section]
\newenvironment{proof}[1][Proof]{\noindent\textbf{#1.} }{\ \rule{0.5em}{0.5em} \medskip }
\newcommand{\ignore}[1]{}

\begin{document}

\begin{center}

{\Large \bf Revisiting Gr\"uss's inequality: covariance bounds, \break 
QDE but not QD copulas, and central moments}

\vspace*{7mm}

\textbf{Mart\'in Egozcue}
\\
\smallskip
\textit{Department of Economics, University of Montevideo,
Montevideo 11600, Uruguay,
and
\\
Accounting and Finance Department, Norte Construcciones,
Punta del Este,
\\
Maldonado 20100, Uruguay.} E-mail: \texttt{megozcue@correo.um.edu.uy}
\bigskip

\textbf{Luis Fuentes Garc\'ia}
\\
\smallskip
\textit{Departamento de M\'etodos Matem\'aticos e de Representaci\'on,
Escola T\'ecnica Superior de Enxe\~neiros de Cami\~nos, Canais e Portos,
Universidade da Coru\~na, 15001 A Coru\~na, Spain.} E-mail: \texttt{lfuentes@udc.es}
\bigskip

\textbf{Wing-Keung Wong}
\\
\smallskip
\textit{Department of Economics and the Institute for Computational Mathematics,
Hong Kong Baptist University, Kowloon Tong, Hong Kong.}
E-mail: \texttt{awong@hkbu.edu.hk}
\bigskip

\textbf{Ri\v{c}ardas Zitikis}\footnote{Corresponding author.
Tel: 519-432-7370; Fax: 519-661-3813; E-mail: zitikis@stats.uwo.ca}
\\
\smallskip
\textit{Department of Statistical and Actuarial Sciences,
              University of Western Ontario, London,
              Ontario N6A 5B7, Canada.}
              E-mail: \texttt{zitikis@stats.uwo.ca}
\end{center}

\newpage

\section*{Abstract}

Since the pioneering work of Gerhard Gr\"uss dating back to 1935,
Gr\"uss's inequality and, more generally, Gr\"uss-type
bounds for covariances have fascinated researchers and found
numerous applications in areas such as economics, insurance, reliability,
and, more generally, decision making under uncertainly.
Gr\"uss-type bounds for covariances have been established mainly under
most general dependence structures, meaning no restrictions
on the dependence structure between the two underlying random variables.
Recent work in the area has revealed a potential for improving
Gr\"uss-type bounds, including the original Gr\"uss's bound,
assuming dependence structures such as quadrant dependence (QD). In this paper
we demonstrate that the relatively little explored notion
of `quadrant dependence in expectation' (QDE)
is ideally suited in the context of bounding covariances,
especially those that appear in the aforementioned areas of application.
We explore this research avenue in detail,
establish general Gr\"uss-type bounds, and illustrate them
with newly constructed examples of
bivariate distributions, which are not QD but, nevertheless, are QDE.
The examples rely on specially devised copulas. We supplement the examples with
results concerning general copulas and their convex combinations.
In the process of deriving Gr\"uss-type bounds, we also establish new bounds
for central moments, whose optimality is demonstrated.

\bigskip
\noindent
{\bf Keywords and phrases}:
Gr\"uss's inequality,
covariance bound,
Hoeffding representation,
Cuadras representation,
quadrant dependence,
quadrant dependence in expectation,
copula,
convex combination,
Archimedean copula,
Fr\'echet copula,
Farlie-Gumbel-Morgenstern copula,
central moments,
Edmundson-Madansky bound.

\newpage

\section{Introduction}
\label{introduction}

The covariance, say $\mathbf{Cov}[V,W]$ between
two random variables $V$ and $W$, has played pivotal roles in
numerous areas such as economics, finance, insurance, statistics,
and, more generally, in decision making under uncertainty. For
details on specific applications with references
to many works in the areas  that have greatly influenced our current research,
we refer to
Broll et al. \cite{BrollEtAl-IMAJMM},
Egozcue et al. \cite{EgozcueEtAl-JIA}, \cite{EgozcueEtAl-IMAJMM},
Furman and Zitikis \cite{furman-zitikis-IME-1},
\cite{furman-zitikis-IME-2}, \cite{furman-zitikis-naaj},
Zitikis \cite{Zitikis}.
A number of mathematics problems, especially those
related to the theory of functions, have also been successfully tackled
with the aid of covariance-type considerations (see, e.g.,
Dragomir and Agarwal \cite{DragomirAgarwal},
Dragomir and Diamond \cite{DragomirDiamond},
Furman and Zitikis \cite{furman-zitikis-jipam}, \cite{furman-zitikis-jmaa},
Izumino and Pe\v{c}ari\'c \cite{IzuminoPecaric},
Izumino et al. \cite{IzuminoPecaricTepes}).
Solutions to problems in these areas often rely on determining the sign
of covariances as well as on establishing their lower and upper bounds.

The random variables $V$ and $W$ are often unobservable but
are known to be transformations (also called distortions) of some
observable random variables $X$ and $Y$; that is,
$V=\alpha(X)$ and $W=\beta(Y)$ for some functions
$\alpha, \, \beta: \mathbf{R} \to \mathbf{R}$. Consequently, the covariance
$\mathbf{Cov}[\alpha(X),\beta(Y)]$ becomes of interest.
In a large number of applications,
only one of the two random variables is distorted. In this paper we
concentrate on this case, thus restricting ourselves to an in-depth analysis
of the covariance
\begin{equation}
\mathbf{Cov}[X,\beta(Y)].
\label{cov-000}
\end{equation}
If compared to the more general covariance $\mathbf{Cov}[\alpha(X),\beta(Y)]$,
this reduction of generality plays a significant role in providing us with
additional technical tools, including the notion of `quadrant
dependence in expectation' (QDE) to be defined in
Section \ref{preliminaries} below, and thus in turn allows us to establish
deeper results than those available in the literature under, say,
the notion of quadrant dependence (QD). In applications where
covariance (\ref{cov-000}) emerges,
the distortion function $\beta $  might be, for example,
a utility or value function (see, e.g.,
Broll et al. \cite{BrollEtAl-IMAJMM},
Egozcue et al. \cite{EgozcueEtAl-JIA}, \cite{EgozcueEtAl-IMAJMM},
and references therein), some insurance-premium loading function
(see, e.g., Furman and Zitikis \cite{furman-zitikis-IME-1},
\cite{furman-zitikis-IME-2}, \cite{furman-zitikis-naaj};
Sendov et al. \cite{SeondovEtAl-new},
and references therein).

When estimating covariance (\ref{cov-000}), perhaps most naturally
that comes first into our mind is the Cauchy-Schwarz inequality
\begin{equation}
\big |\mathbf{Cov}[X,\beta(Y)] \big | \le \sqrt{\mathbf{Var}[X]}
\,\sqrt{\mathbf{Var}[\beta(Y)]}\, ,
\label{cov-cauchy}
\end{equation}
where $\mathbf{Var}[X]$ is the variance (i.e., $\mathbf{Cov}[X,X]$)
of the random variable $X$. Furthermore, assuming that there are
finite intervals $[a,A]$ and $[b,B]$ such that
$X\in [a,A]$ and $\beta(Y)\in [b,B]$ almost surely,
from bound (\ref{cov-cauchy}) we immediately obtain
Gr\"uss's \cite{Gruss} inequality
\begin{equation}
\big |\mathbf{Cov}[X,\beta(Y)] \big | \le {(A-a)(B-b)\over 4}
\label{cov-5e}
\end{equation}
(see, e.g., Zitikis \cite{Zitikis} for details and references).
Inequalities (\ref{cov-cauchy}) and (\ref{cov-5e}) hold irrespectively
of the dependence structure
between $X$ and $\beta(Y)$, which implies that the inequalities also
hold under the `worst possible' dependence scenario,
which is associated with the strongest dependence structure
between $X$ and $\beta(Y)$, arising when $X=\beta(Y)$ almost surely.
It is under this scenario that the optimality of the Gr\"uss's bound
has been established in the literature, and we refer to,
e.g., Dragomir \cite{Dragomir}, \cite{Dragomir-book},
Mitrinovi\'c et al. \cite{Mitrinovi},
Steele \cite{Steele},
and Zitikis \cite{Zitikis} for further
notes, examples, and references on the topic.

When the random variables $X$ and $\beta(Y)$
are independent, which in particular happens when
the underlying random variables $X$ and $Y$ are such,
then the covariance $\mathbf{Cov}[X,\beta(Y)]$ is zero.
Hence, knowing how much and in what sense the random variables
$X$ and $Y$ are dependent plays a significant role
when investigating the magnitude
of the covariance $\mathbf{Cov}[X,\beta(Y)]$ and its sign,
among other properties.
This line of research has been advocated by Zitikis \cite{Zitikis} and
Egozcue \textit{et al.} \cite{EgozcueEtAl-JIA}, who have employed the notion
of quadrant dependence to be defined rigorously in
Section \ref{preliminaries} below.

We conclude this section with a guide through the rest of this paper.
In Section \ref{problem}, we first show how the assumption of bivariate
normality leads, via the well-known Stein's Lemma, towards
a Gr\"uss-type covariance bound. We then extend this bivariate normal case into
the formulation of a general Gr\"uss-type covariance bound,
which we aim at establishing in various situations
throughout the current paper.
In Section \ref{preliminaries}, we recall definitions of
QD and QDE and their counterparts for copulas,
and also relate these notions of dependence 
to Gr\"uss-type covariance bounds. In Section \ref{preliminaries}
we also establish general results concerning convex mixtures of
negative quadrant dependent (NQD) and positive quadrant dependent (PQD)
copulas that provide
a basis for constructing bivariate distributions which are QDE but not QD.
We devote Section \ref{copulas} to constructing several illustrative examples of
copulas which are QDE but not QD; as far as we are aware of, these
examples are the first ones in the literature.
In Section \ref{qde-bounds}, we establish QDE-based Gr\"uss-type
bounds for covariance (\ref{cov-000}), discuss their optimality
and highlight the importance of having tight bounds for
central moments of random variables.
We investigate the latter bounds in great detail
in Section \ref{central-moments}.
Since the QDE notion of dependence also naturally leads towards
regression-based considerations, in Section \ref{regression} we
establish regression-based Gr\"uss-type
bounds for covariance (\ref{cov-000}).

\section{Formulation of the problem}
\label{problem}

Applications often suggest models for $(X,Y)$
but it may not be feasible to assume models for the pair $(X,\beta(Y))$ because
the distortion function $\beta $ may change depending on, say,
investor, insurer, etc. For this reason 
it is desirable to separate the underlying
stochastic model, which is based on $(X,Y)$,
from the class of distortion functions $\beta $.

Stein \cite{Stein} noted that if the pair $(X,Y)$
follows the bivariate normal distribution and
the function $\beta$ is differentiable, then
\begin{equation}
\mathbf{Cov}[X,\beta(Y)]=\mathbf{Cov}[X,Y] \mathbf{E}[\beta'(Y)].
\label{cov-10stein}
\end{equation}
This equation, frequently known as Stein's Lemma, separates the dependence structure
from the distortion function $\beta $. For extensions and
generalizations of this result, we refer to
Furman and Zitikis \cite{furman-zitikis-IME-1},
\cite{furman-zitikis-IME-2},
\cite{furman-zitikis-naaj},
\cite{furman-zitikis-astin}, and references therein.
In particular, it has been observed that
equation (\ref{cov-10stein}) is a direct consequence of the following one
\begin{equation}
\mathbf{Cov}[X,\beta(Y)]
={\mathbf{Cov}[X,Y]\over \mathbf{Var}[Y]}\mathbf{Cov}[Y,\beta(Y)],
\label{cov-10i}
\end{equation}
which separates the dependence structure of $(X,Y)$
from the distortion function $\beta $ but does not
require the differentiability of $\beta $.

Now we rewrite equation (\ref{cov-10i}) in the form
\begin{equation}
\mathbf{Cov}[X,\beta(Y)]
=\mathbf{Corr}[X,Y]\, \mathbf{G}_0[X,Y,\beta ],
\label{cov-10ii}
\end{equation}
which we call to be of the `Gr\"uss form' for reasons to be made clear below
(Problem \ref{mainproblem} below),
where $\mathbf{Corr}[X,Y]$ is the Pearson correlation coefficient
between $X$ and $Y$, and $\mathbf{G}_0[X,Y,\beta ]$ is a `Gr\"uss factor'
defined by
\[
\mathbf{G}_0[X,Y,\beta ]=
\sqrt{{\mathbf{Var}[X]\over \mathbf{Var}[Y]}}\, \mathbf{Cov}[Y,\beta(Y)].
\]
Note that the Gr\"uss factor $\mathbf{G}_0[X,Y,\beta ]$
does not depend on the bivariate distribution of $(X,Y)$ except that it depends on
the cumulative distribution functions (cdf) $F$ and $G$ of the underlying random
variables $X$ and $Y$, respectively, and also on the distortion function $\beta $.
By the Cauchy-Schwarz inequality, $|\mathbf{G}_0[X,Y,\beta ]|$ does not exceed
the product of the standard deviations
$\sqrt{\mathbf{Var}[X]}$ and $\sqrt{\mathbf{Var}[\beta(Y)]}$,
which do not exceed $(A-a)/2$ and $(B-b)/2$, respectively,
under what we call the `Gr\"uss condition':
\begin{itemize}
\item
There are two finite intervals $[a,A]\subset \mathbf{R}$ and
$[b,B]\subset \mathbf{R}$ such that $X\in [a,A]$
and $\beta(Y)\in [b,B]$ almost surely.
\end{itemize}
Hence, under the Gr\"uss condition,
we have that $\mathbf{G}_0[X,Y,\beta ]$ does not exceed
the right-hand side of
bound (\ref{cov-5e}), and we thus have that
\begin{equation}
\big | \mathbf{Cov}[X,\beta(Y)]\big |
\le \big | \mathbf{Corr}[X,Y]\big | \, {(A-a)(B-b)\over 4}.
\label{cov-10iii}
\end{equation}
Since $|\mathbf{Corr}[X,Y]|$ does not exceed $1$,
bound (\ref{cov-10iii}) implies Gr\"uss's bound (\ref{cov-5e})
irrespectively of the dependence structure
between $X$ and $Y$. When these two random variables
are independent, then the right-hand side of
bound (\ref{cov-5e}) is zero. This demonstrates
the pivotal role of the dependence structure when 
sharpening Gr\"uss's bound.

Reflecting upon the notes above,
we next put forward a general formulation of the problem that we shall tackle
from various angles throughout this paper.

\begin{problem}\label{mainproblem}
We are interested in establishing bounds of the form
\begin{equation}
\big | \mathbf{Cov}[X,\beta(Y)]\big |
\le \mathbf{D}[X,Y]\, \mathbf{G}[X,Y,\beta ],
\label{cov-10iv}
\end{equation}
where
\begin{itemize}
\item
$\mathbf{D}[X,Y]$ is a `dependence coefficient', which
must be equal to $0$ when the random variables $X$ and $Y$ are independent, 
and should not depend on the distortion function $\beta $; 
\item
$\mathbf{G}[X,Y,\beta ]$ is a `Gr\"uss factor', which should not depend on
the dependence structure between $X$ and $Y$ but may depend on $\beta $ and
the cdf's  $F$ and $G$ of $X$ and $Y$, respectively.
\end{itemize}
\end{problem}

Throughout the paper we assume that the distortion function
$\beta: \mathbf{R} \to \mathbf{R}$ is
of bounded variation, meaning that it can be written as
the difference $\beta=\beta_1-\beta_2$ of two non-decreasing functions
$\beta_1, \beta_2: \mathbf{R} \to \mathbf{R}$.
The corresponding function $|\beta|: \mathbf{R} \to \mathbf{R}$
is defined by the equation $|\beta|(y)= \beta_1(y)+\beta_2(y)$.
When $\beta $ is differentiable, then $d|\beta|(y)=|\beta'(y)|dy$.
Furthermore, we use $\mathbf{1}\{S\}$
for the indicator function of statement $S$ which
is equal to $1$ when the statement $S$ is true and $0$ otherwise.
Hence, in particular, for any random variable $Z$ and any real number $z$,
\[
\tau_z(Z)=\mathbf{1}\{Z>z\}
\]
is a random variable that takes on the value $1$ when $Z>z$
and $0$ otherwise. We shall frequently view $\tau_z(Z)$
as a random function of $z$. In our following considerations,
we shall also use the sign-function, $\mathrm{sign}(x)$, 
which takes on three values:
$-1$ when $x<0$, $0$ when $x=0$, and $+1$ when $x>0$.

\section{QD and QDE random variables and copulas}
\label{preliminaries}

One of the most fundamental equations that we utilize in the present paper
is the Cuadras-Hoeffding representation
\begin{equation}
\mathbf{Cov}[\alpha(X),\beta(Y)]
= \iint \mathbf{Cov}\big [\tau_x(X),\tau_y(Y)\big ] d\alpha(x)\, d\beta(y)
\label{cov-general}
\end{equation}
of the covariance between the transformed random variables
$\alpha(X)$ and $\beta(Y)$. The representation has been established by
Cuadras \cite{Cuadras} assuming, naturally and necessarily, 
that the expectations of $\alpha(X)$, $\beta(Y)$, and $\alpha(X)\beta(Y)$
are well-defined and finite. Covariance representation (\ref{cov-general})
generalizes the classical Hoeffding's \cite{Hoeffding} representation
established in the case $\alpha(x)=x$ and $\beta(x)=x$ (see also Sen \cite{Sen}).
The importance of representation (\ref{cov-general}) in our context
is that it achieves a separation of the dependence structure
present in $(X,Y)$ from the distortion functions $\alpha $ and $\beta $.
Hence, in particular, the positive quadrant-dependence (definition follows) implies
that $\mathbf{Cov}[X,Y]\ge 0$, and the negative quadrant-dependence
implies that $\mathbf{Cov}[X,Y]\le 0$. These are, of course,
well-known facts (Lehmann \cite{Lehmann}).

\begin{definition}[\it Lehmann \cite{Lehmann}]\label{quad-dep}
Two random variables $X$ and $Y$ are
\textit{positively (resp. negatively) quadrant dependent}
if $\mathbf{Cov}\big [\tau_x(X),\tau_y(Y)\big ]\ge 0$
(resp. $\le 0$) for all $x,y\in \mathbf{R}$. We abbreviate
this as PQD (resp. NQD), and when it is not important to specify whether
the two random variables are PQD or NQD, then we simply say that
they are quadrant dependent (QD).
\end{definition}

As a special case of representation (\ref{cov-general}) we have
the following one:
\begin{equation}
\mathbf{Cov}[X,\beta(Y)]
= \iint \mathbf{Cov}\big [\tau_x(X),\tau_y(Y)\big ] dx\, d\beta(y).
\label{cov-general-spec-1}
\end{equation}
Note that the inner integral on the right-hand
side of equation (\ref{cov-general-spec-1}) is equal to
$\mathbf{Cov}\big [X,\tau_y(Y)\big ]$, and so
representation (\ref{cov-general-spec-1}) becomes
\begin{equation}
\mathbf{Cov}[X,\beta(Y)]
= \int \mathbf{Cov}\big [X,\tau_y(Y)\big ] d\beta(y).
\label{cov-general-spec-2}
\end{equation}
The integrand on the right-hand side of equation (\ref{cov-general-spec-2})
is related to the following definition.

\begin{definition}[\it Kowalczyk and Pleszczynska \cite{Kowalczyk-Pleszczynska}]
\label{quad-dep-e}
A random variable $X$ is
\textit{positively (resp. negatively) quadrant dependent in expectation}
on a random variable $Y$
if $\mathbf{Cov}\big [X,\tau_y(Y)\big ]\ge 0$
(resp. $\le 0$) for all $y\in \mathbf{R}$.  We abbreviate
this as $X$ is PQDE (resp. NQDE) on $Y$, and when
it is not important to specify whether
these two random variables are PQDE or NQDE, then we simply say that
$X$ is quadrant dependent in expectation (QDE) on $Y$.
\end{definition}

QDE is not a stronger notion than QD, which follows from the
already noted but not explicitly written equation:
\begin{equation}
\mathbf{Cov}\big [X,\tau_y(Y)\big ]
=
\int \mathbf{Cov}\big [\tau_x(X),\tau_y(Y)\big ]dx.
\label{cov-0-general-i}
\end{equation}
For discussions and
hints on potential applications of this notion of dependence,
we refer to Kowalczyk and Pleszczynska \cite{Kowalczyk-Pleszczynska},
Wright \cite{wright}, and references therein. 
One would actually expect that 
QDE is a weaker notion than QD, which means that there must be pairs $(X,Y)$
which are QDE (i.e., either NQDE or PQDE) but not QD (i.e, neither NQD nor PQD).
Our search of the literature has not, however, revealed examples that would
formally confirm this non-equivalence of QDE and QD. Hence,
we next present general results pointing in the direction of non-equivalence,
and we shall use them in Section \ref{copulas} as our guide when constructing
specific examples of bivariate distributions that are QDE but not QD.

The main tool that we are going to employ is the notion of copula, which
is a surface $(u,v)\mapsto C(u,v)$ defined on
the square $[0,1]\times [0,1]$ and such that
$\mathbf{P}[X\le x, Y\le y]$ is equal to $C(F(x),G(y))$,
where $F$ and $G$ are the cdf's of $X$ and $Y$, respectively.
Hence, in particular, we have the equation
\begin{equation}
\mathbf{Cov}\big [X,\tau_y(Y)\big ]
=\int \Big (C(F(x),G(y))- F(x)G(y)\Big ) dx.
\label{cov-iii}
\end{equation} 
When the random variables $X$ and $Y$ have uniform (marginal) distributions,
then we denote them by  $U$ and $V$, respectively. In turn, we have the following
reformulations of Definitions \ref{quad-dep} and \ref{quad-dep-e} in terms
of the copula $C$, which is connected to the bivariate distribution
of $(U,V)$ via the equation
\[
\mathbf{P}[U\le u, V\le v]=C(u,v).
\]
Namely, $U$ and $V$ are PQD (resp. NQD) if
$C(u,v)\ge uv$ (resp. $C(u,v)\le uv$) for all $u,v\in [0,1]$,
and $U$ is PQDE (resp. NQDE) on $V$
if $\mathcal{C}(v)\ge 0$ (resp. $\mathcal{C}(v)\le 0$)
for all $v\in [0,1]$, where
$\mathcal{C}(v)=\mathbf{Cov}\big [U,\tau_v(V)\big ]$,
that is (cf.\, equation (\ref{cov-iii})),
\[
\mathcal{C}(v)=\int_0^1 \big (C(u,v)- uv\big ) du .
\]
In general, we have the following QD and QDE definitions for copulas.

\begin{definition}\label{copula-quad-dep}
Copula $(u,v)\mapsto C(u,v)$ is PQD (resp. NQD) if
$C(u,v)\ge uv$ (resp. $C(u,v)\le uv$) for all $u,v\in [0,1]$.
The copula is QD if it is either NQD or PQD.
\end{definition}

\begin{definition}\label{copula-quad-dep-e}
Copula $(u,v)\mapsto C(u,v)$ is PQDE (resp. NQDE)
if $\mathcal{C}(v)\ge 0$ (resp. $\le 0$) for all $v\in [0,1]$.
The copula is QDE if it is either NQDE or PQDE.
\end{definition}

\begin{note}\rm 
In Definition \ref{copula-quad-dep-e} it would be more
precise to say that $U$ is PQDE (resp. NQDE) on $V$
if $\mathcal{C}(v)\ge 0$ (resp. $\le 0$) for all $v\in [0,1]$.
Analogously, $U$ is QDE on $V$ if $U$ is either NQDE or PQDE on $V$.
We avoid this pedantry by always considering the `first variable' to be
(N/P)QDE on the `second variable'.
\end{note}

Hence, the problem that we are interested in at the moment
is whether there are any copulas that are
QDE (i.e., either NQDE or PQDE) but not QD (i.e, neither NQD nor PQD).
The following two general theorems are fundamental in solving this problem,
with illustrative examples provided in Section \ref{copulas}.

\begin{theorem}
\label{general}
Let $C_{0}(u,v)$ and $C_{1}(u,v)$ be NQD and PQD copulas, respectively.
Denote their convex combination by
$C_\alpha(u,v)=(1-\alpha)C_0(u,v)+\alpha C_1(u,v)$ with parameter
$\alpha \in[0,1]$. Suppose that the surface
\begin{equation}
(u,v)\mapsto \frac{uv-C_0(u,v)}{C_1(u,v)-C_0(u,v)}
\label{funct-cond-1}
\end{equation}
is not constant on $[0,1]\times [0,1]$.
Then there exist $m,m',M',M\in [0,1]$ such that $0\leq m\leq m'$, $M'\leq M\leq 1$,
and $m<M$, and such that the copula $C_\alpha$ is: 
\begin{itemize}
\item
NQD for $\alpha\in [0,m]$;
\item
neither NQD nor PQD for $\alpha\in (m,M)$;
\item
PQD for $\alpha\in [M,1]$;
\item
NQDE if and only if $\alpha \in [0,m']$ (it could be that $m=m'$);
\item
neither NQDE nor PQDE for $\alpha\in (m',M')$ (it could be that $m'\geq M'$,
in which case the interval $(m',M')$ is empty);
\item
PQDE if and only if $\alpha \in [M',1]$ (it could be that $M=M'$).
\end{itemize}
\end{theorem}

\begin{proof} Let
$I_{-}=\{\alpha\in[0,1]\,:\, C_\alpha\hbox{ is NQD}\}$
and
$I_{+}=\{\alpha\in[0,1]\,:\,  C_\alpha\hbox{ is PQD}\}$.
We have the following facts:
\begin{enumerate}
\item
$0\in I_-$ and $1\in I_+$.
\item $I_-$ (similarly $I_+$) is a closed subspace of $[0,1]$.
Namely, if
$C_{\alpha_k}(u,v)-uv\leq 0$ for all $ u,v\in [0,1]$
and $\alpha_k\to \alpha$, then
$C_{\alpha}(u,v)-uv\leq 0$ for all $ u,v\in [0,1]$.
\item
$I_-$ (similarly $I_+$) is a connected space. Namely, if $\alpha,\beta \in I_-$,
then $C_\gamma$ for any $\gamma\in [\alpha,\beta]$ is a convex
combination of $C_\alpha $ and $C_\beta$, and so it is NQD.
\item
$I_-$ and $I_+$ are closed intervals (it follows from 2 and 3).
\item
$I_-\cap I_+=\emptyset$. We prove this by contradiction.
Suppose that there exists $\alpha\in I_-\cap I_+$. Then $C_\alpha$ is NQD and PQD.
This implies that
$(1-\alpha)C_0(u,v)+\alpha C_1(u,v)-uv=0$ for all $ u,v\in [0,1]$.
Hence, function (\ref{funct-cond-1}) is equal to
the constant $\alpha $; a contradiction.
\end{enumerate}
In view of the above facts we have that $I_-=[0,m]$ and $I_+=[M,1]$ with $m<M$,
and the first three statements of Theorem \ref{general} follow. In a similar way,
but working with the function
\begin{equation}\label{function-ccc}
\mathcal{C}_{\alpha}(v)=\int_0^1 \big (C_{\alpha}(u,v)- uv\big ) du,
\end{equation}
we establish the other three statements of Theorem \ref{general}.
Note that NQD (PQD) implies
NQDE (PQDE), and so we must have $m\leq m'$ and $M'\leq M$.
This completes the proof of Theorem \ref{general}.
\end{proof}

\begin{note}\rm
If we have $m<m'$, then there are $\alpha$ values such that
the copula $C_\alpha$
is neither NQD nor PQD, but it is NQDE. Similarly, if we have $M'<M$, then
there are $\alpha$ values such that the copula $C_\alpha$
is neither NQD nor PQD, but it is PQDE.
\end{note}

\begin{theorem}\label{general-2}
Let the assumptions of Theorem \ref{general} be satisfied,
and let $\mathcal{C}_{\alpha}(v)$ be given by equation (\ref{function-ccc}).
Furthermore, assume that there is a constant $\kappa \in [0,1]$ such that
\[
\frac{\mathcal{C}_0(v)}{\mathcal{C}_0(v)-\mathcal{C}_1(v)}=\kappa
\]
for all $v\in (0,1)$. Then there is an open interval of $\alpha$ values
such that the copula $C_\alpha$ is neither NQD nor PQD,
but it is either NQDE or PQDE.
\end{theorem}

\begin{proof} We have that
$\mathcal{C}_{\kappa} (v)=0$ for all $v\in [0,1]$.
Thus, the copula $C_{\kappa }$ is both PQDE and NQDE.
By the previous theorem, we have that
$M'\leq \kappa \leq m'$. Since $m\leq m'$, $M'\leq M$ and $m<M$,
we deduce that $m<m'$ or $M'<M$. Consider these two cases separately:
1) if $m<m'$, then for any $\alpha\in (m,m']$ the copula $C_\alpha$ is
neither NQD nor PQD, but it is NQDE, and
2) if $M'<M$, then for any $\alpha\in [M',M)$ the copula $C_\alpha$ is
neither NQD nor PQD, but it is PQDE.
This completes the proof of Theorem \ref{general-2}.
\end{proof}

\section{Examples of QDE copulas which are not QD}
\label{copulas}

Here we give three examples of QDE copulas that are not QD.
In the first two examples we
choose NQD and PQD copulas such that their convex combinations
are not QD but, nevertheless, are QDE. The third example is based
on a copula which is not QD but, under an appropriate choice
of marginal distributions, produces a bivariate distribution
that is not QD but, nevertheless, is QDE. These three
examples open up broad avenues for constructing
QDE copulas that are not QD, using a myriad of existing copulas
whose QD-type properties have been documented
in the literature (e.g., Nelsen \cite{nelsen}).
For discussions concerning copulas in the context of
actuarial, financial, and other applications,
we refer to, for example, Denuit et al. \cite{Denuit},
Genest and Favre \cite{gf-2007},
Genest et al. \cite{ggb-2009},
McNeil et al. \cite{McNeil-Frey-Embrechts},
and references therein.

\begin{example}\label{example-19}\rm
The Fr\'echet lower-bound (FL) copula
\[
C_{FL}(u,v)=\max\{0,u+v-1\}
\]
is NQD, and the Fr\'echet upper-bound (FU) copula
\[
C_{FU}(u,v)=\min\{u,v\}
\]
is PQD. Both are defined on the unit square $[0,1]\times [0,1]$.
Let $C_{\alpha}$ be the convex combination of the two Fr\'echet copulas
(cf.\ McNeil et al. \cite{McNeil-Frey-Embrechts}):
\begin{equation}\label{conv-1}
C_{\alpha}(u,v)=(1-\alpha) C_{FL}(u,v) + \alpha C_{FU}(u,v),
\end{equation}
where $\alpha \in (0,1)$.
We see from Figure \ref{figure-00a}
\begin{figure}[h!]
\bigskip
\begin{center}
\includegraphics[width=5cm]{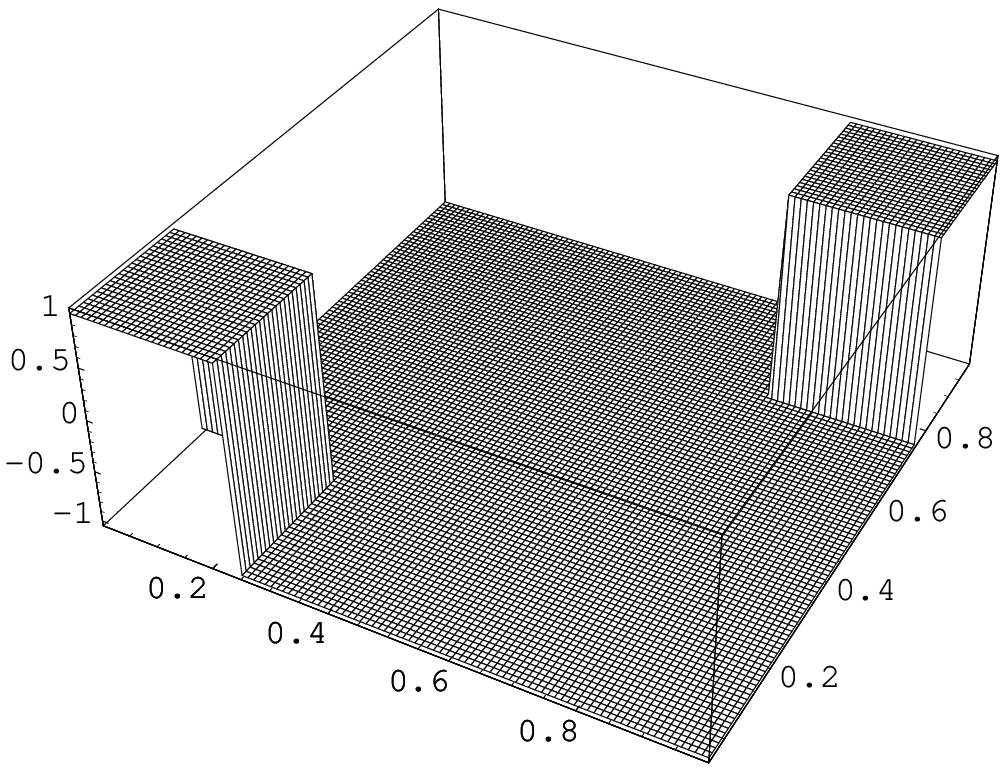}
\includegraphics[width=5cm]{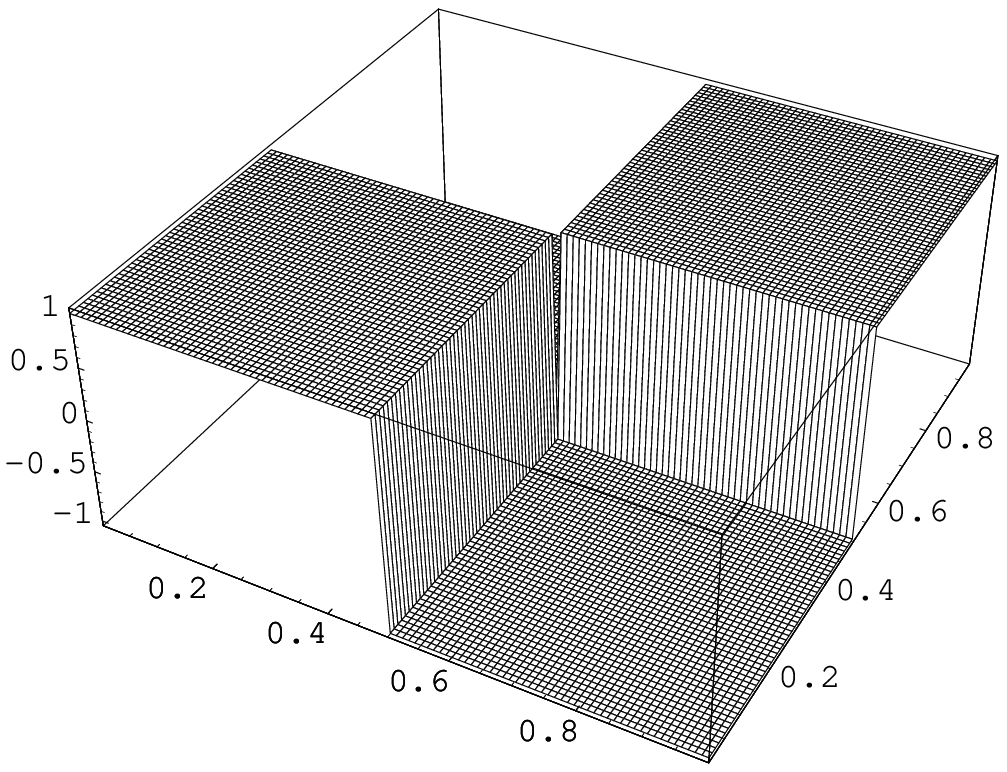}
\includegraphics[width=5cm]{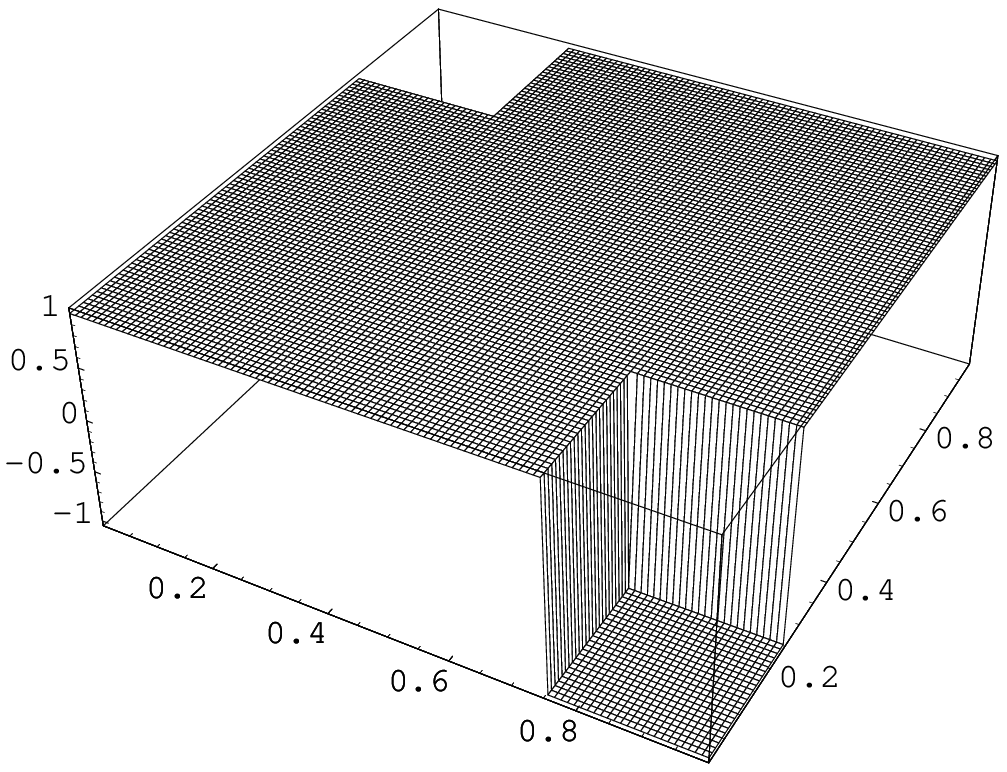} \\ \bigskip
\includegraphics[width=5cm]{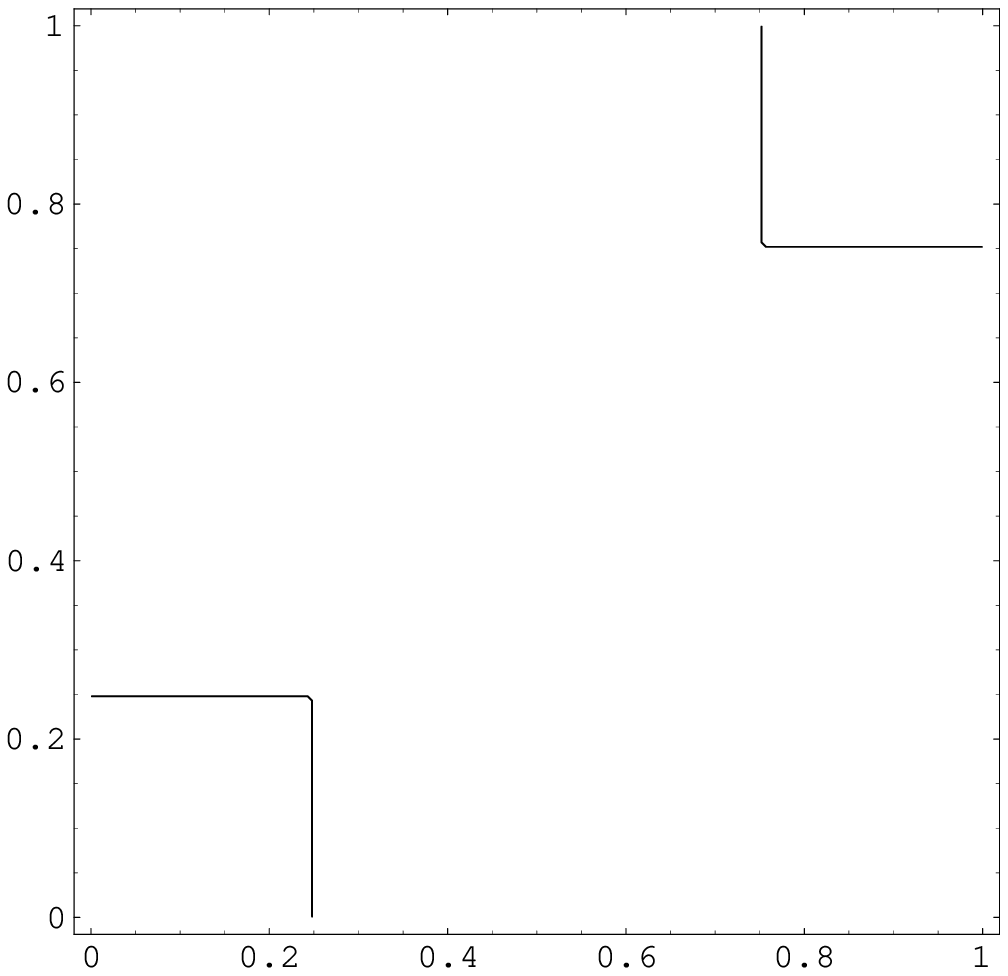}
\includegraphics[width=5cm]{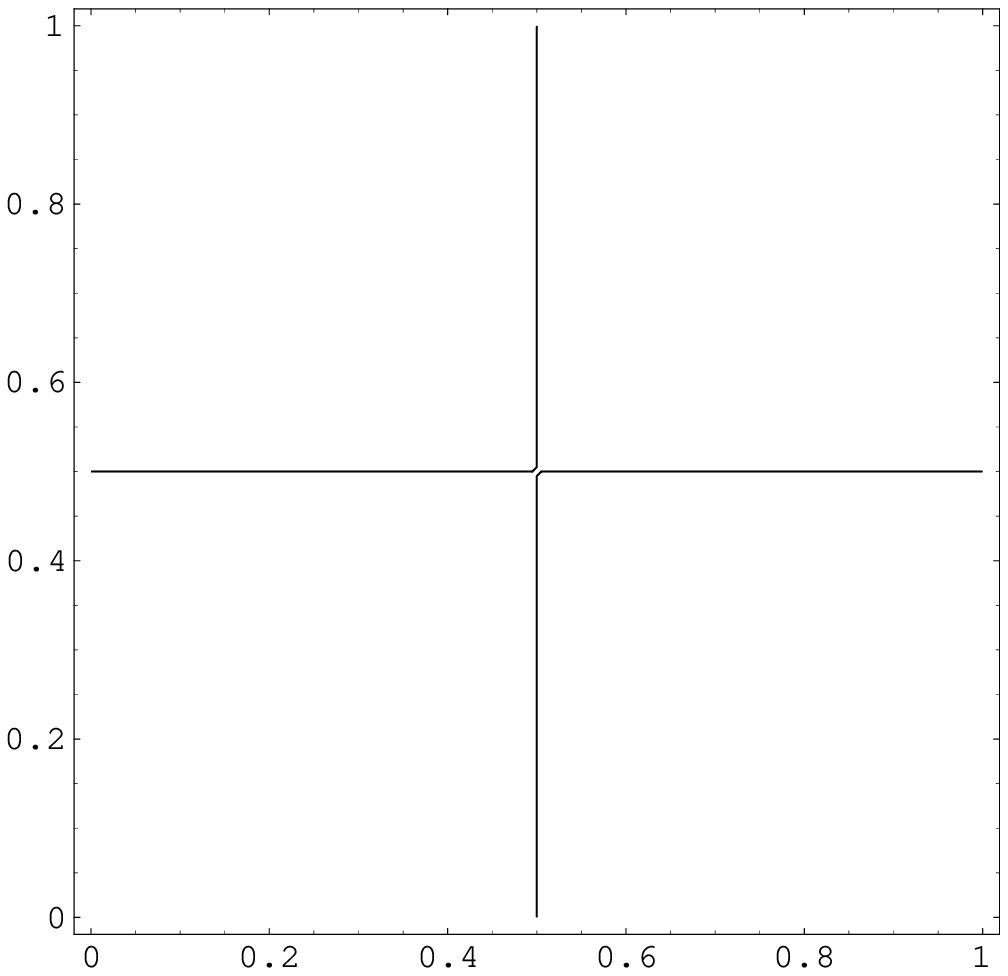}
\includegraphics[width=5cm]{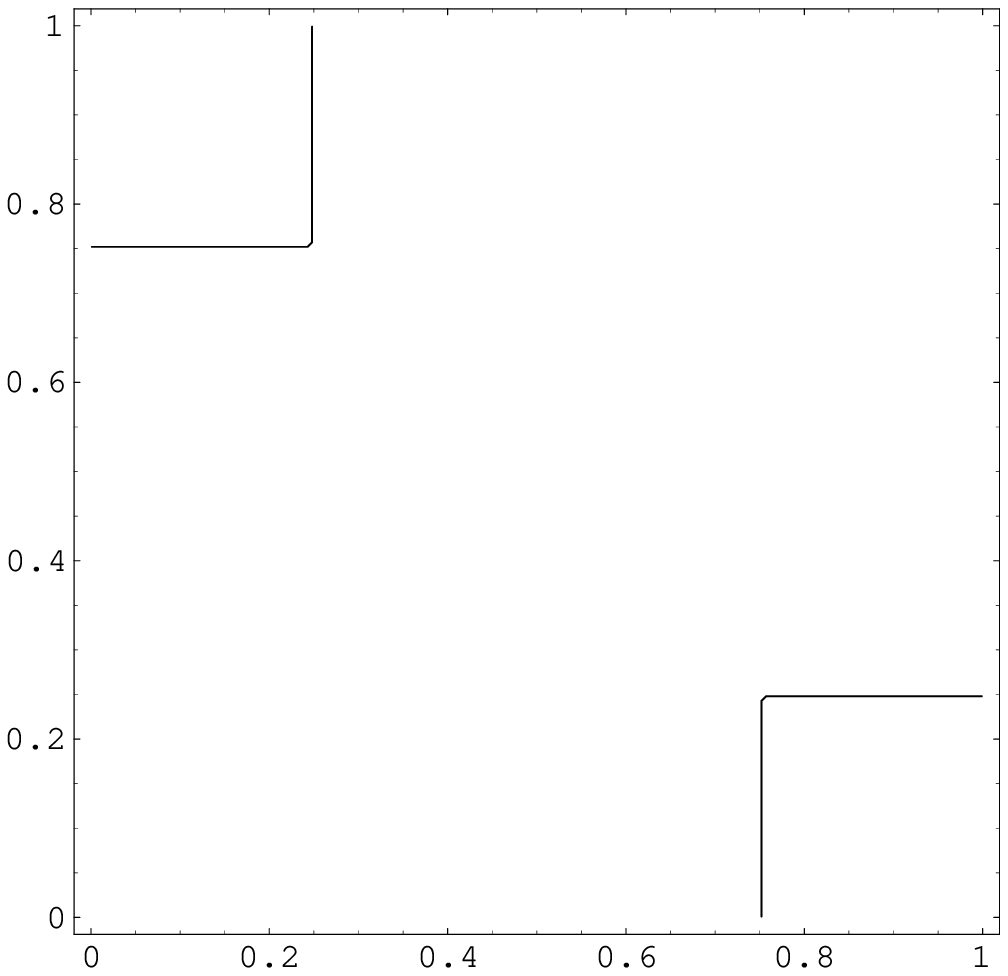}
\end{center}
\caption{The surface $(u,v)\mapsto \mathrm{sign}[C_{\alpha}(u,v)-uv]$ when
$\alpha =1/4$ (left top panel),
$\alpha =1/2$ (center top panel),
and $\alpha =3/4$ (right top panel) with
the corresponding contour plots beneath them.}
\label{figure-00a}
\end{figure}
that the copula $C_{\alpha}$ is neither PQD nor NQD.
To check whether $C_{\alpha}$
is QDE (i.e., either PQDE or NQDE), we calculate the integral
\begin{equation}
\mathcal{C}_{\alpha}(v)
=\int_0^1 \big (C_{\alpha}(u,v)- uv\big ) du
= v(1-v)\bigg ( \alpha - {1\over 2} \bigg ).
\label{c-1}
\end{equation}
Hence, $\mathcal{C}_{\alpha}(v)\le 0$ for all $v\in [0,1]$
if and only if $\alpha \le 1/2$, meaning that the copula $C_{\alpha}$ is NQDE.
Likewise, $\mathcal{C}_{\alpha}(v)\ge 0$ for all $v\in [0,1]$
if and only if $\alpha \ge 1/2$, meaning that $C_{\alpha}$ is PQDE.
Hence, for example, when $\alpha =1/4$, then $C_{\alpha}$
is neither PQD nor NQD but it is NQDE. Likewise, when
$\alpha =3/4$, then $C_{\alpha}$ is neither PQD nor NQD but it is PQDE.
This concludes Example \ref{example-19}.
\end{example}

\begin{example}\label{example-20}\rm
Here we first choose the Farlie-Gumbel-Morgenstern (FGM) copula
\[
C_{FGM}(u,v)=uv\big (1+\theta (1-u)(1-v)\big )
\]
with $\theta \in [-1,1]$;
we set the parameter $\theta $ to $-1$ throughout this example
to make the FGM copula NQD. Next we choose the already noted
Fr\'echet upper-bound copula $C_{FU}(u,v)=\min\{u,v\}$,
which is PQD. Let $\alpha \in (0,1)$ be a parameter, and
let $C_{\alpha}$ be the convex combination of the above two copulas:
\begin{equation}\label{conv-2}
C_{\alpha}(u,v)=(1-\alpha ) C_{FGM}(u,v) + \alpha C_{FU}(u,v).
\end{equation}
We have that
\begin{align}
C_{\alpha}(u,v)-uv
&= \alpha \big (\min\{u,p\}-uv \big )-(1-\alpha ) uv(1-u)(1-v)
\notag
\\
&=
\left\{
  \begin{array}{ll}
    u(1-v)\big ( 1- (1-\alpha ) (1+v(1-u)) \big ) & \hbox{when }\, u\le v, \\
    v(1-u)\big ( 1- (1-\alpha ) (1+u(1-v)) \big ) & \hbox{when }\, u\ge v.
  \end{array}
\right.
\label{curve-UL}
\end{align}
Hence, $C_{\alpha}(u,v)-uv \ge 0$ for only those
$(u,v)\in [0,1]\times [0,1]$ that are between
(cf.\, equation (\ref{curve-UL})) the zero-curve
\[
U_{\alpha}=\big \{ (u,v):\, (1-\alpha ) (1+v(1-u))=1 \big \}
\quad \bigg [ \textrm{ that is, } v=v_U(u)={1\over 1-u} \, {\alpha \over 1-\alpha }  \bigg ]
\]
from above, and the zero-curve
\[
L_{\alpha}=\big \{ (u,v):\, (1-\alpha ) (1+u(1-v))=1 \big \}
\quad \bigg [ \textrm{ that is, } v=v_L(u)=1-{1\over u} \, {\alpha\over 1-\alpha }  \bigg ]
\]
from below.
We illustrate the two curves in Figure \ref{figure-01}.
\begin{figure}[h!]
\bigskip
\begin{center}
\includegraphics[width=5cm]{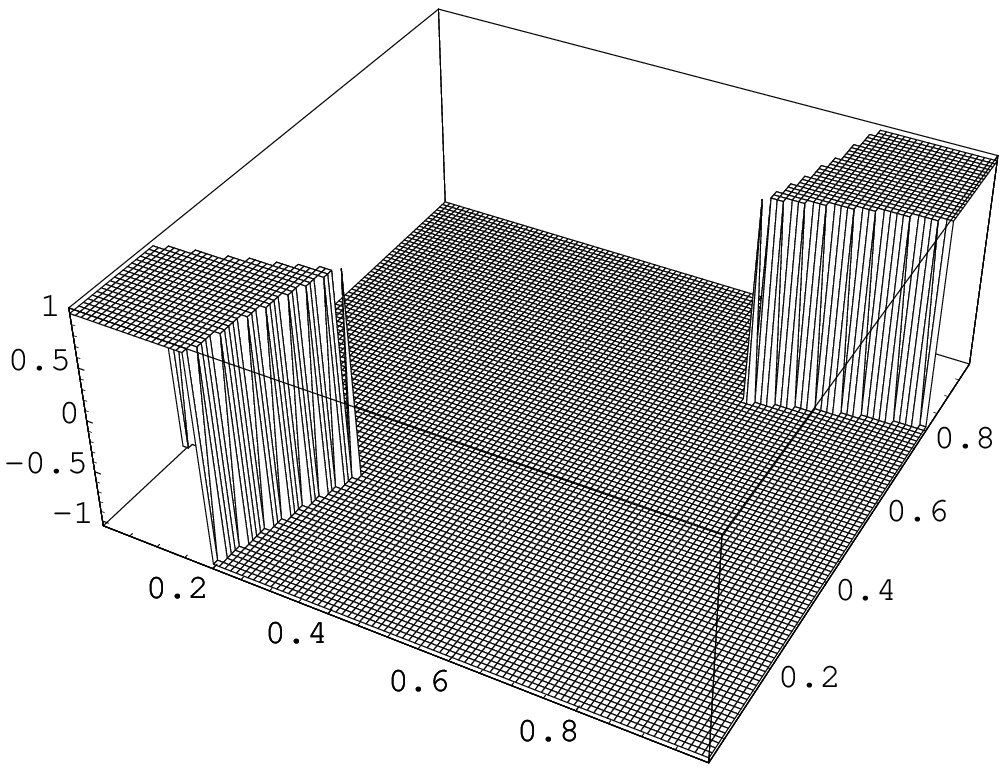}
\includegraphics[width=5cm]{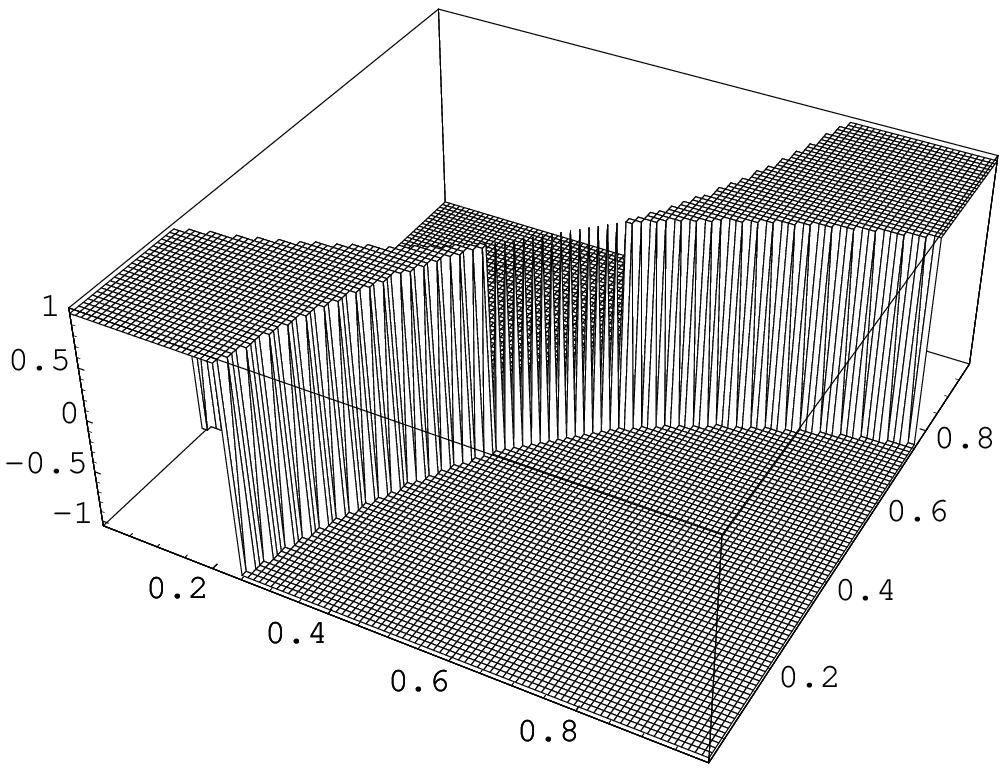}
\includegraphics[width=5cm]{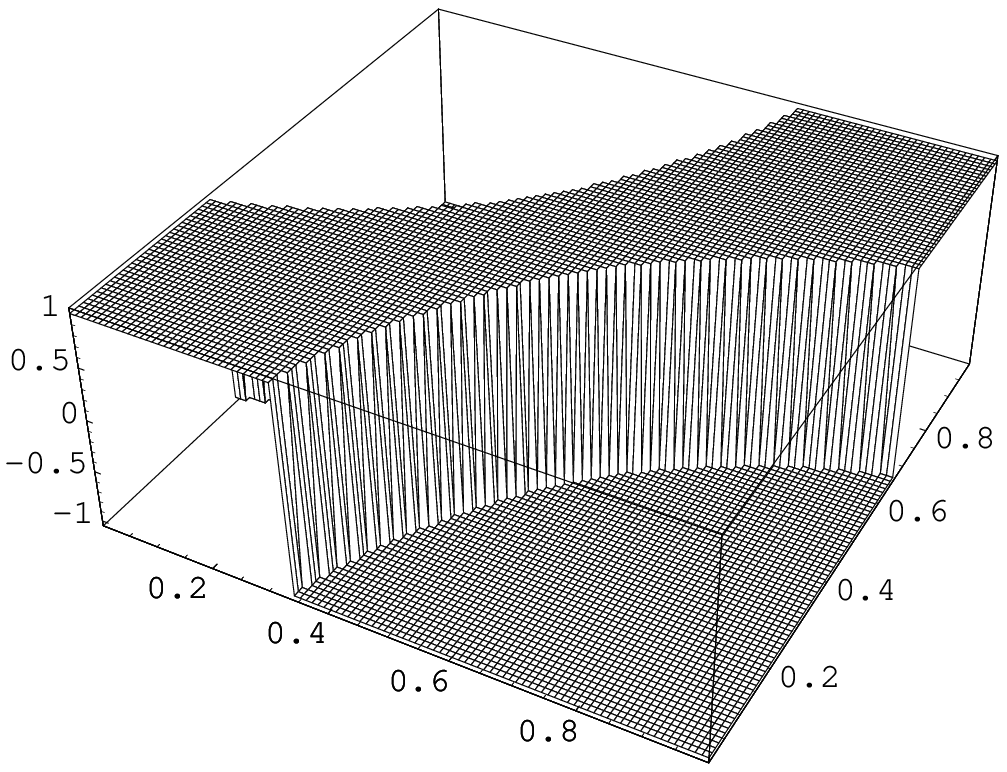} \\ \bigskip
\includegraphics[width=5cm]{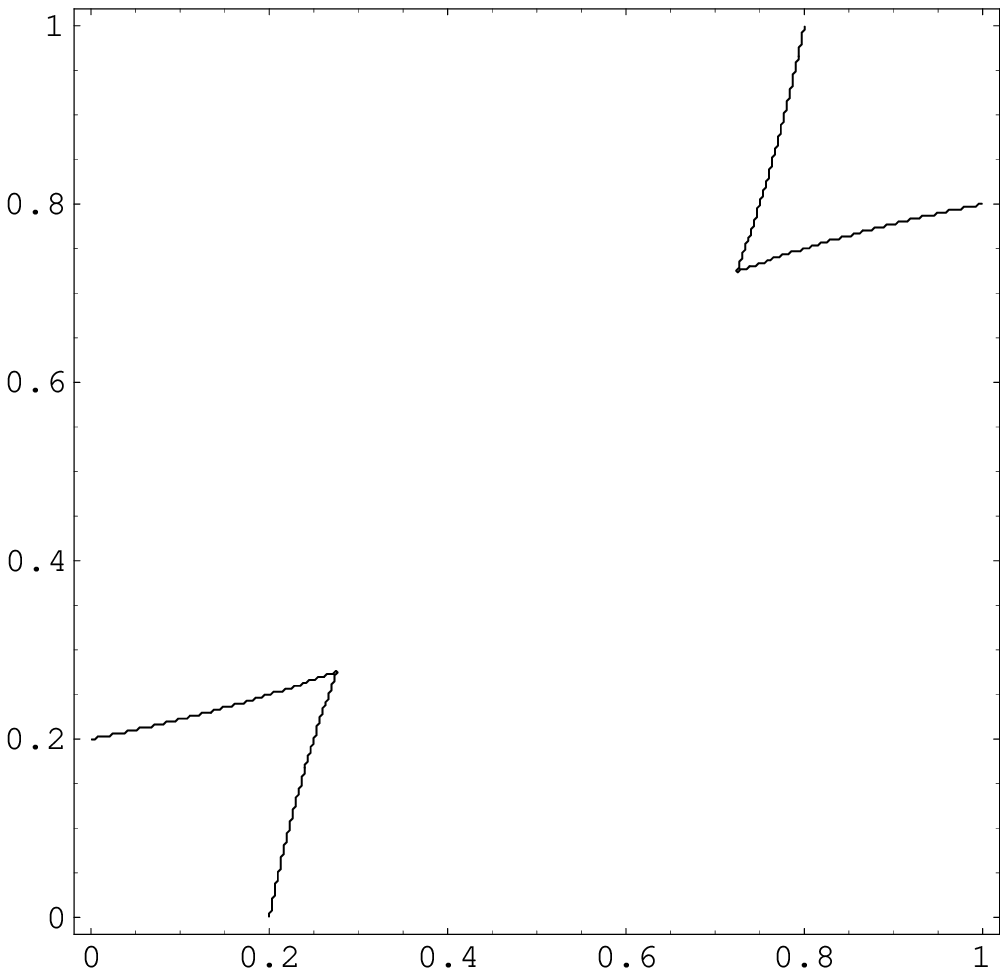}
\includegraphics[width=5cm]{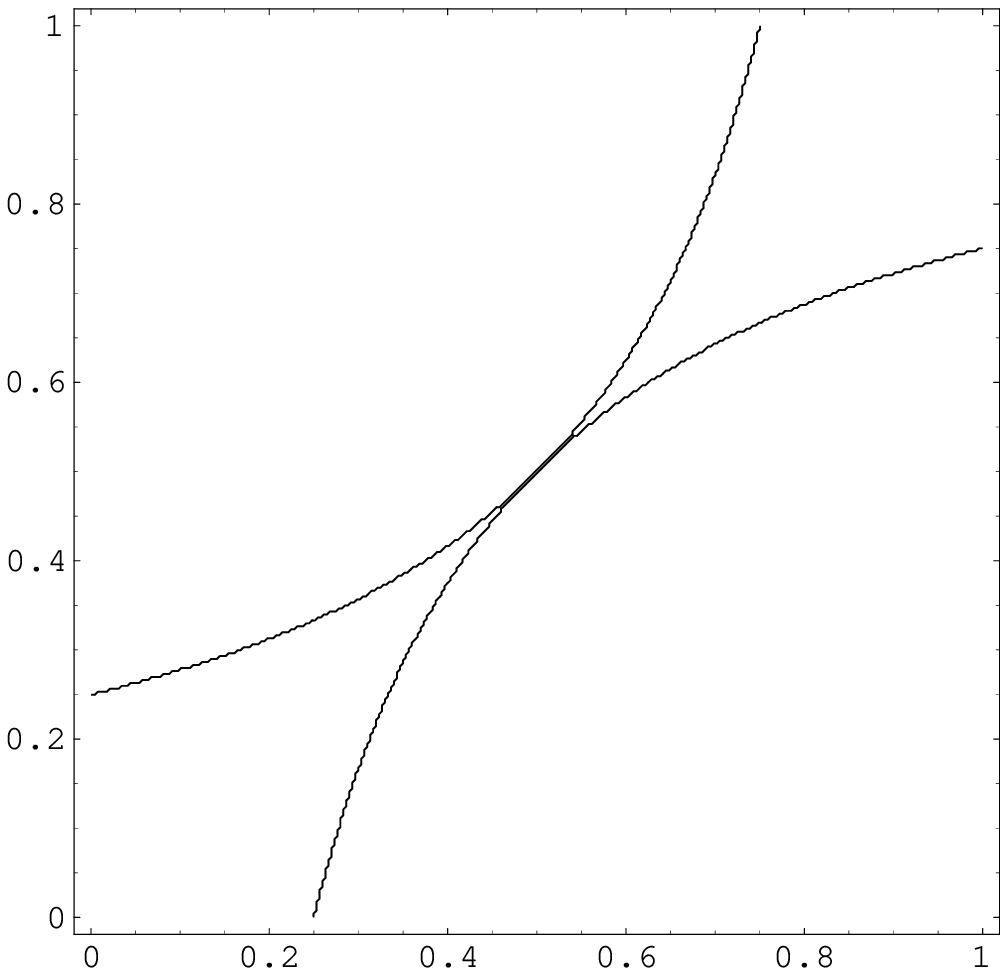}
\includegraphics[width=5cm]{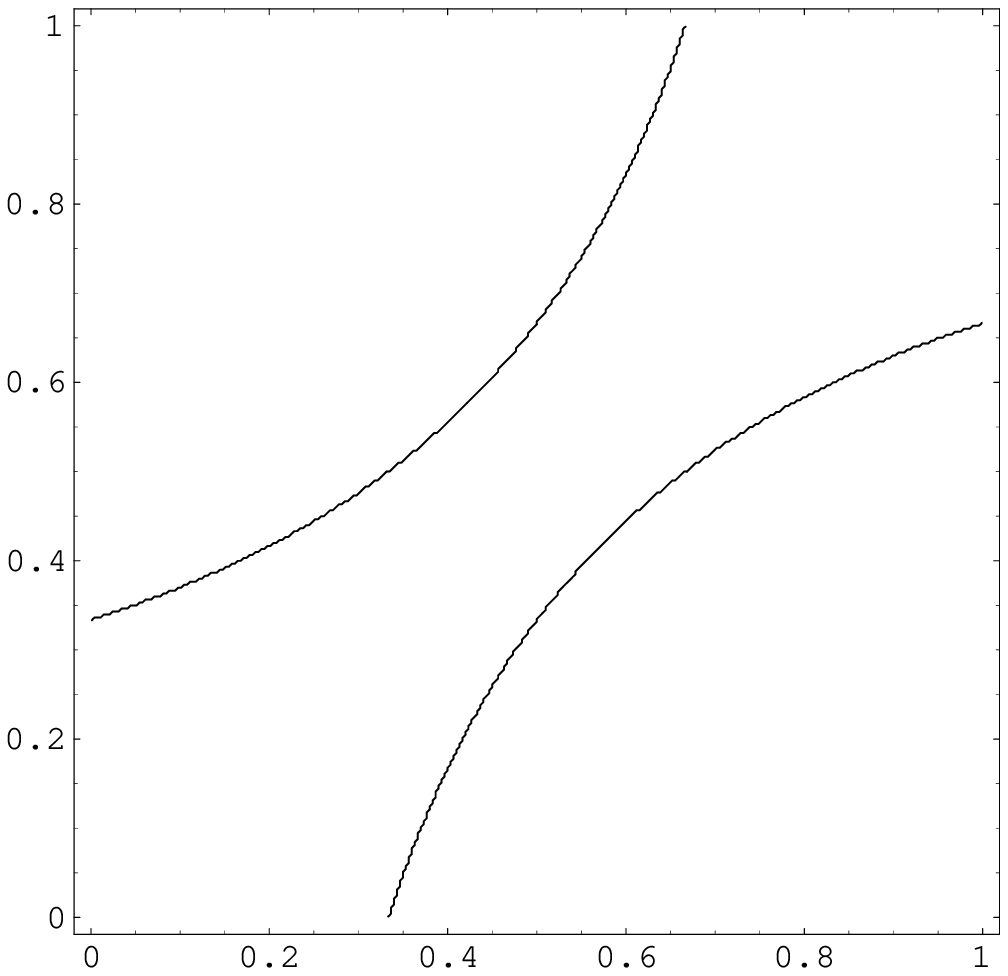}
\end{center}
\caption{The surface $(u,v)\mapsto \mathrm{sign}[C_{\alpha}(u,v)-uv]$ when
$\alpha =1/6$ (left top panel),
$\alpha =1/5$ (center top panel),
and $\alpha =1/4$ (right top panel) with
the zero-curves $U_{\alpha}$ and $L_{\alpha}$
in the corresponding plots beneath them.}
\label{figure-01}
\end{figure}
Note that the curves intersect in the interior of
the square $[0,1]\times [0,1]$ only when $\alpha \in (0, 1/5)$ and
touch each other at one point when $\alpha =1/5$.
As to the PQDE or NQDE, we calculate the integral
\begin{equation}
\mathcal{C}_{\alpha}(v)
=\int_0^1 \big (C_{\alpha}(u,v)- uv\big ) du
= {v(1-v)\over 2}\bigg ( {4\over 3}\, \alpha -{1\over 3}\bigg ).
\label{c-2}
\end{equation}
Hence, $\mathcal{C}_{\alpha}(v)\le 0$ for all $v\in (0,1)$
meaning that $C_{\alpha}$ is NQDE if and only if $\alpha \le 1/4$, and
$\mathcal{C}_{\alpha}(v)\ge 0$ for all $v\in (0,1)$
meaning that $C_{\alpha}$ is PQDE if and only if $\alpha \ge 1/4$.
This concludes Example \ref{example-20}.
\end{example}

\begin{example}\label{example-21}\rm
We model the pair $(X,Y)$ using the following
Archimedean copula (Genest and MacKay \cite{Genest--MacKay-CJS},
\cite{Genest--MacKay-AS};  Genest and Ghoudi \cite{Genest-Ghoudi};
see also Nelsen \cite{nelsen} for additional information and references)
\[
C_{\alpha}(u,v)=\max \left \{0, 1-\left (
(1-u^{\alpha })^{1/\alpha }+ (1-v^{\alpha })^{1/\alpha }
\right )^{\alpha }\right \}^{1/\alpha },
\]
where $\alpha \in (0,1)$ is parameter.
The copula $C_{\alpha}(u,v)$ is not QD.
The zero-curve
$Z_{\alpha}(u,v)=C_{\alpha}(u,v)-uv=0$, which separates
the NQD region from the PQD region, is given by
\[
Z_{\alpha}=\big \{ (u,v):\,
(1-u^{\alpha })^{1/\alpha }+ (1-v^{\alpha })^{1/\alpha }
=(1-u^{\alpha }v^{\alpha } )^{1/\alpha } \big \},
\]
depicted in Figure \ref{figure-1}.
\begin{figure}[h!]
\bigskip
\begin{center}
\includegraphics[width=5cm]{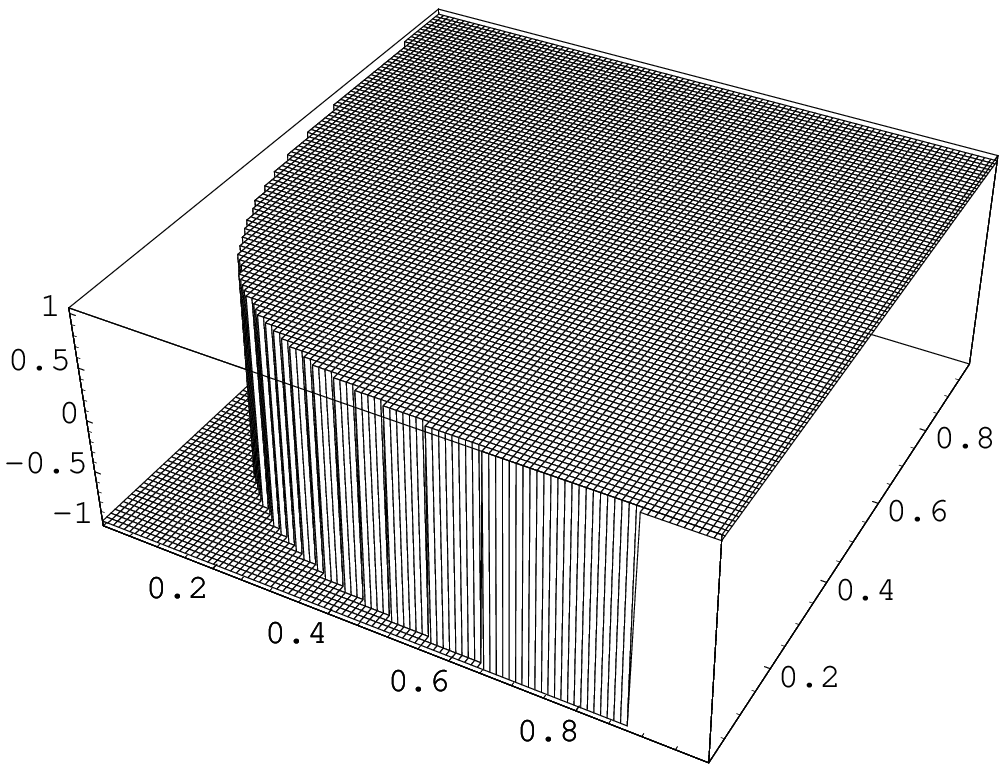}
\includegraphics[width=5cm]{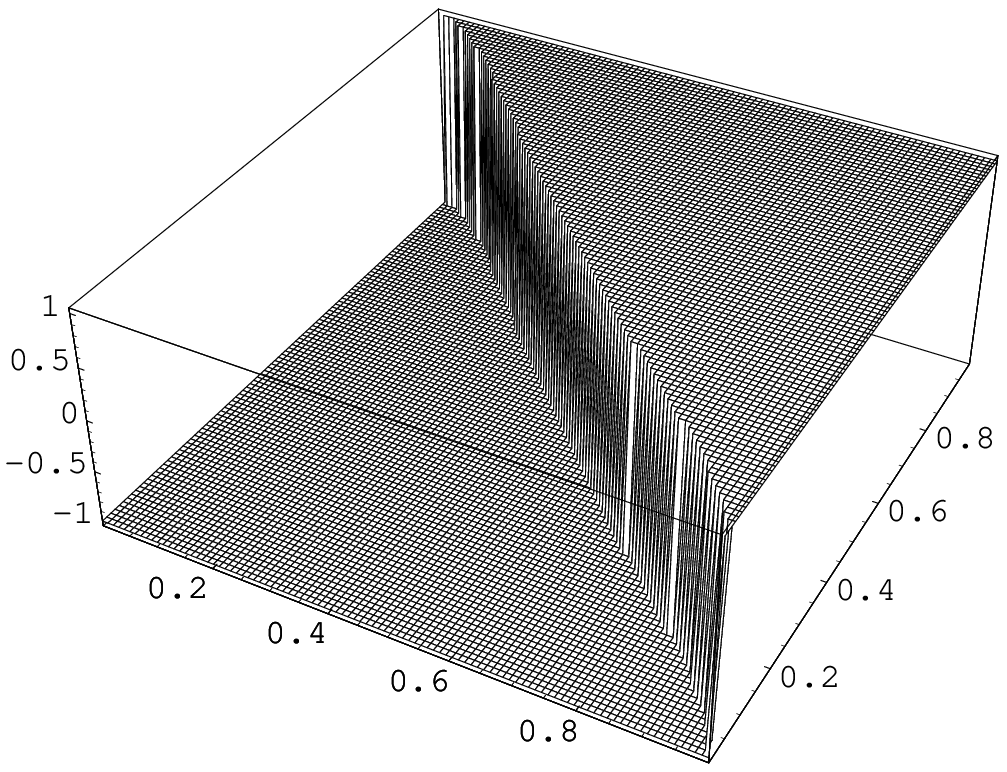}
\includegraphics[width=5cm]{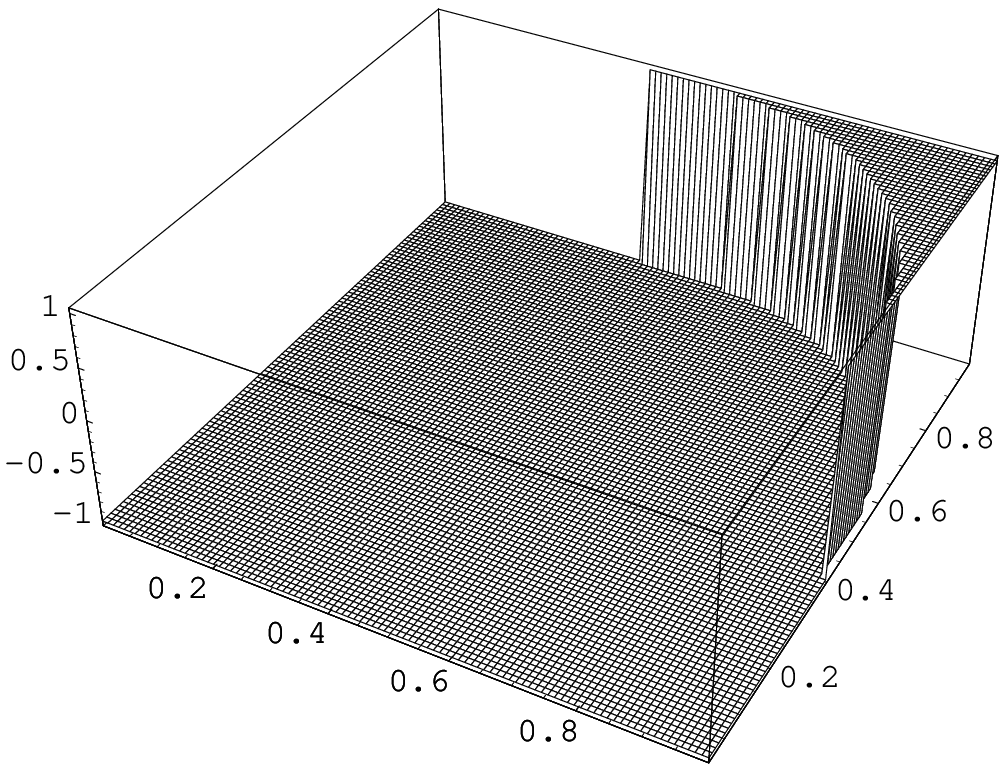} \\ \bigskip
\includegraphics[width=5cm]{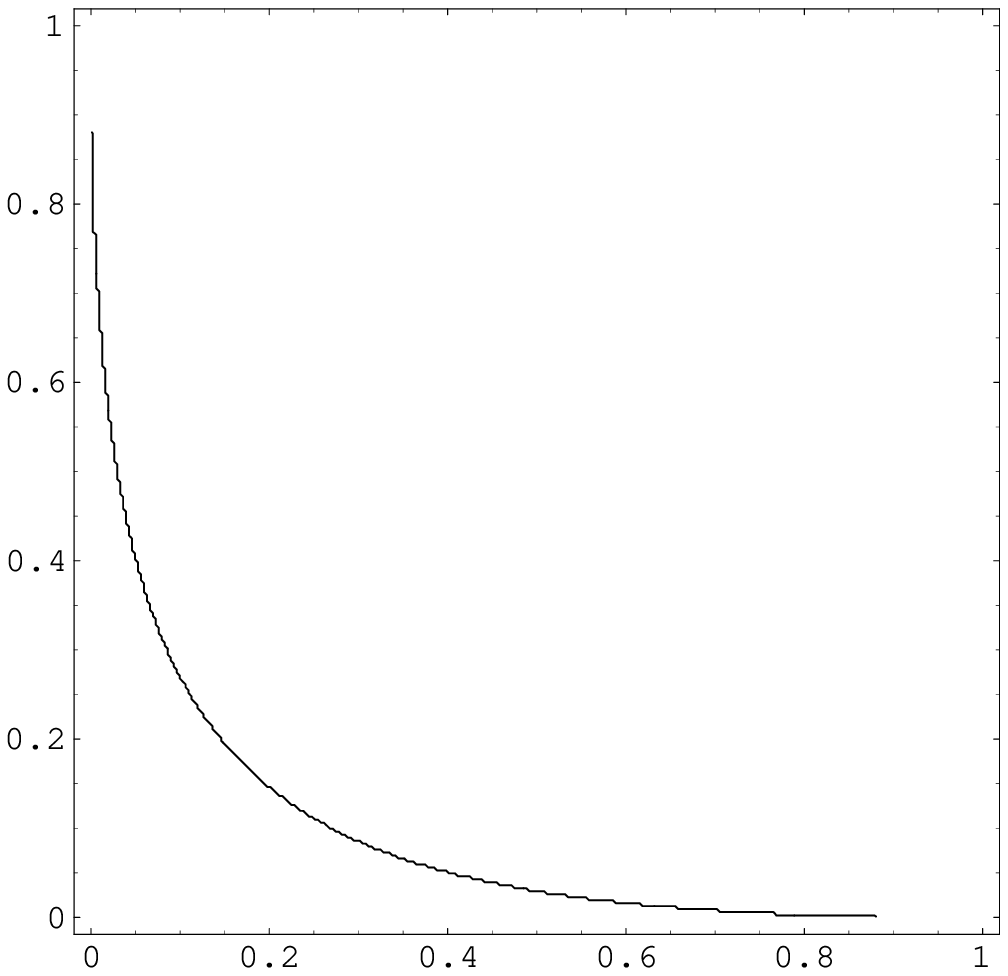}
\includegraphics[width=5cm]{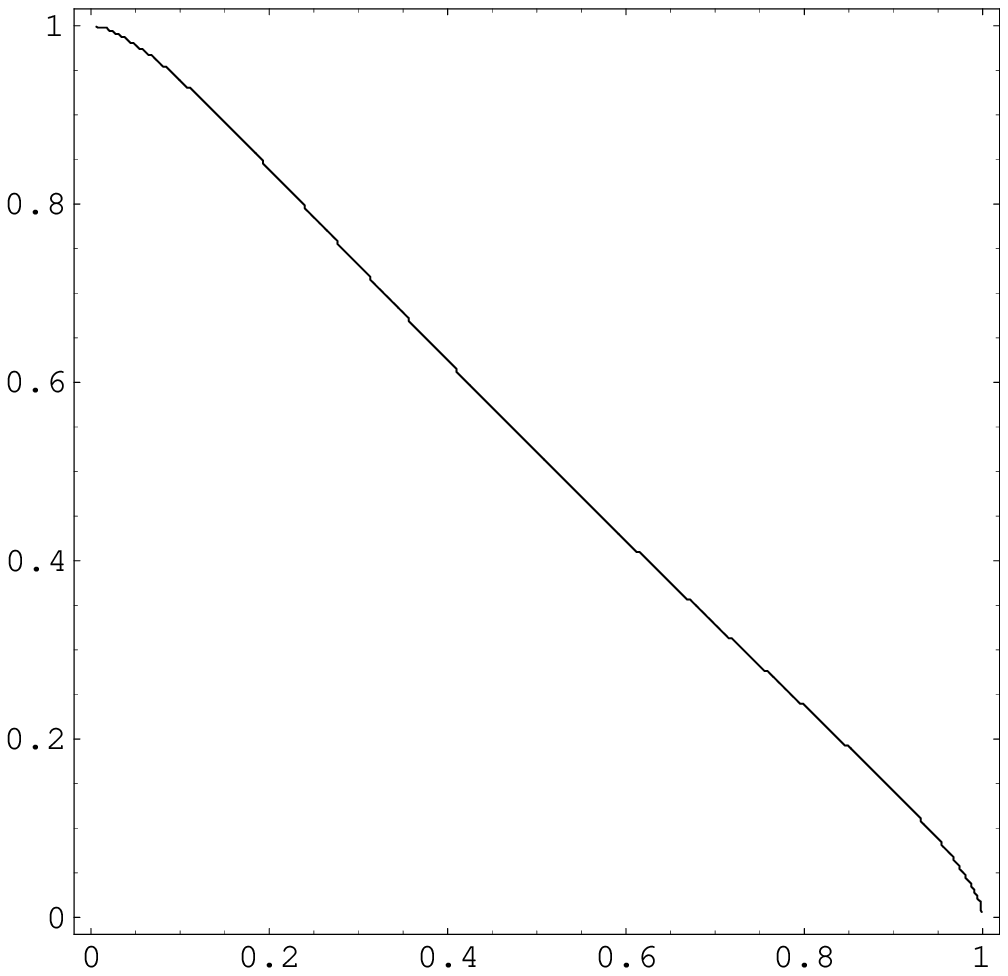}
\includegraphics[width=5cm]{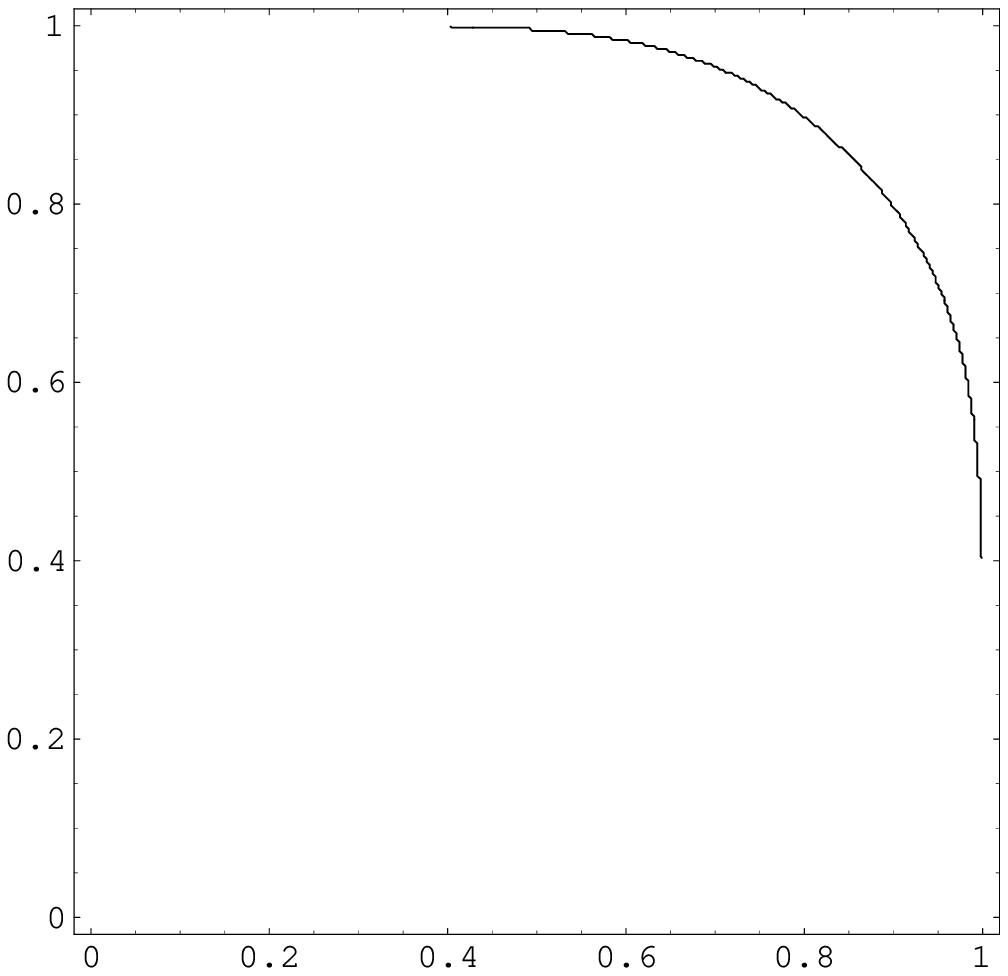}
\end{center}
\caption{The surface $(u,v)\mapsto \mathrm{sign}[C_{\alpha}(u,v)-uv]$ when
$\alpha =5/10$ (left top panel),
$\alpha =7/10$ (center top panel),
and $\alpha =9/10$ (right top panel) with
the zero-curve $Z_{\alpha}$ in the corresponding plots beneath them.}
\label{figure-1}
\end{figure}
Concerning the QDE property, we want to know if and
when the function $\mathcal{C}_{\alpha}$ defined by the formula
\[
\mathcal{C}_{\alpha}(v)=\int_0^1 \big (C_{\alpha}(u,v)- uv\big ) du
\]
is positive or negative. Our experimental analysis
has revealed that this function is negative for some $v\in [0,1]$
and positive for other $v$, and this applies to 
every $\alpha \in (0,)]$. For an illustration,
we have produced Figure \ref{figure-2}.
\begin{figure}[h!]
\bigskip
\begin{center}
\includegraphics[width=5cm]{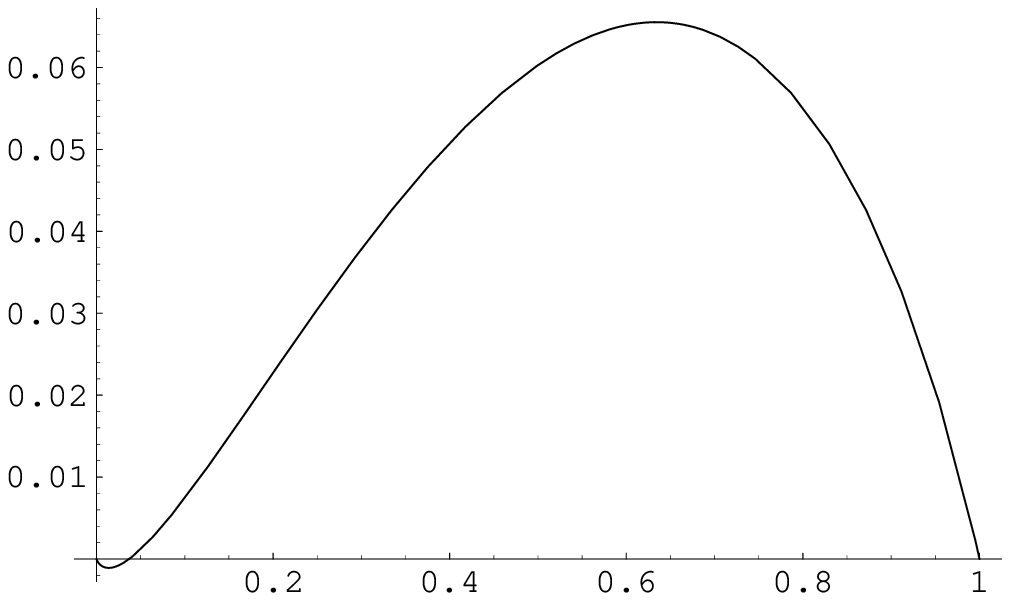}
\includegraphics[width=5cm]{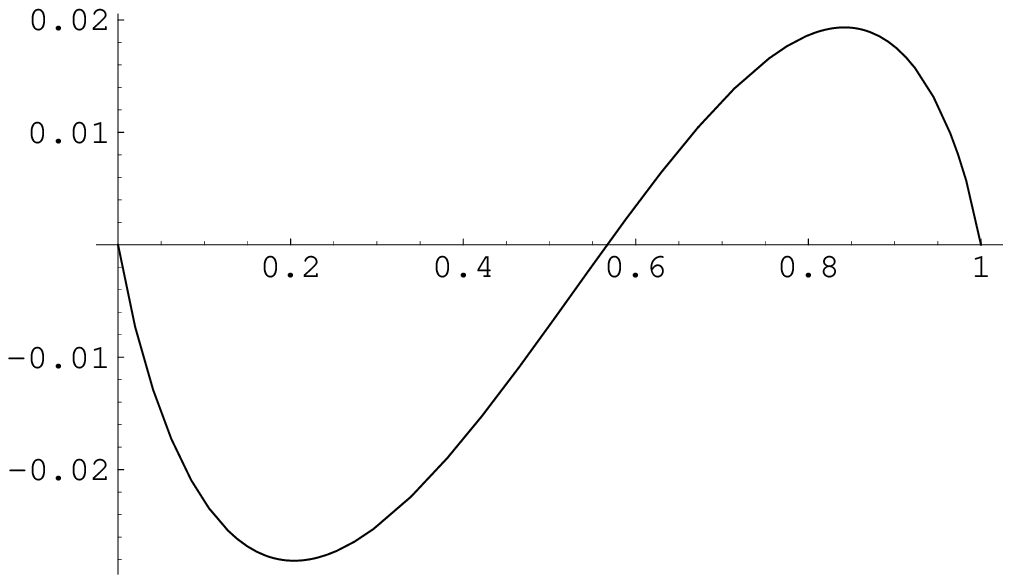}
\includegraphics[width=5cm]{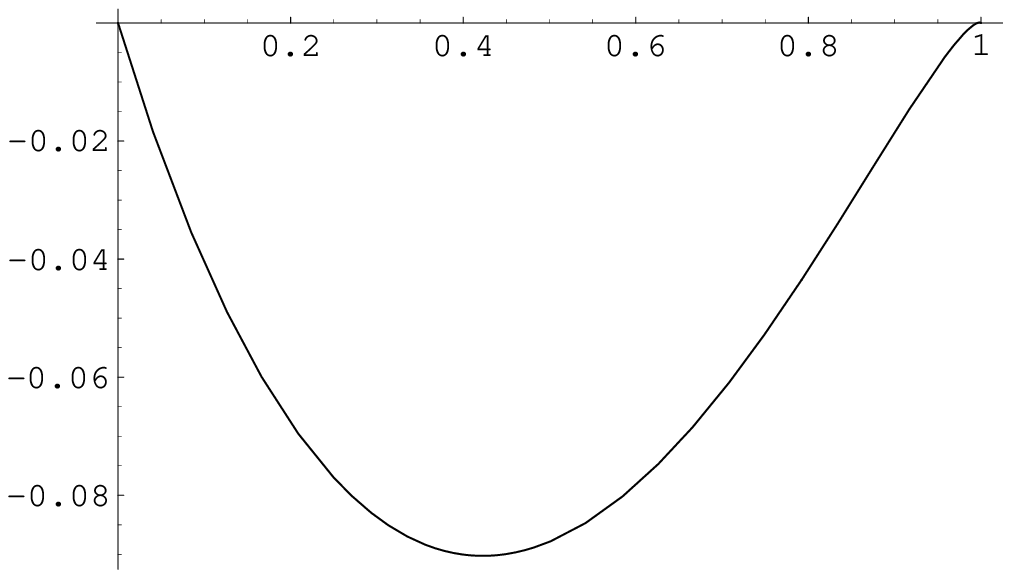}
\end{center}
\caption{The function $v\mapsto \mathcal{C}_{\alpha}(v)$ when
$\alpha =5/10$ (left panel), $\alpha =7/10$ (center panel),
and $\alpha =9/10$ (right panel). In every case, there are regions
(not necessarily clearly visible in the graphs) where the function
is strictly positive and strictly negative.}
\label{figure-2}
\end{figure}
This implies that for the purpose of constructing QDE pairs of
random variables we cannot choose both marginal cdf's uniform. Hence,
we choose only the first marginal cdf
(i.e., that of $X$) uniform. Since the second marginal cdf
(i.e., that of $Y$) cannot be uniform, nor any other continuous cdf,
we construct a discrete cdf $G$ or, equivalently, a discrete
random variable $Y$ in such a way that $U (=X)$ is PQDE on $Y$.

First we note that, for every $v\in (0,1)$,
\begin{itemize}
  \item
$C_{\alpha}(u,v)- uv$ is positive for
some $u\in (0,1)$ and negative for some other $u\in (0,1)$, thus violating
the QD property.
\end{itemize}
To verify this non-QD property rigorously, we first rewrite the copula
$C_{\alpha}(u,v)$ as $\max\{0,P(u,v)\}^{1/{\alpha}}$ with the notation
\[
P(u,v)=1-((1-u^{\alpha})^{1/{\alpha}}+(1-v^{\alpha})^{1/{\alpha}})^{\alpha}.
\]
For every fixed $ v$ and when $u=0$, then we have
$P(0,v)=1-(1+(1-v^{\alpha})^{1/{\alpha}})^{\alpha}<0$.
Thus, $ C_{\alpha}(u,v)=0$ in a neighbourhood of $u=0$.
From this we deduce that $ C_{\alpha}(u,v)-uv<0$ when
$u\in (0,\epsilon)$ for some $\epsilon >0$.
In a neighbourhood of $u=1$, we have that
$C_{\alpha}(u,v)-uv$ is equal to $P(u,v)^{1/{\alpha}}-uv$,
which is a differentiable function.
The partial derivative $(d/du)(P(u,v)^{1/{\alpha}}-uv)$ at the point
$u=1$ is negative. Hence, the function $u\mapsto P(u,v)^{1/{\alpha}}-uv$
is decreasing in a neghbourhood of $1$.
When $u=1$, then we have $P(1,v)^{1/{\alpha}}=v$
and thus $P(1,v)^{1/{\alpha}}-1\cdot v=0$. Hence,
the function $u\mapsto P(u,v)^{1/{\alpha}}-uv$ or, equivalently,
$u\mapsto C_{\alpha}(u,v)-uv$ is positive in a neighbourhood 
to the left of $u=1$. This establishes
the property formulated under the bullet above.

Next, in view of the equation
\begin{align}
\mathbf{Cov}\big [U,\tau_y(Y)\big ]=\mathcal{C}_{\alpha}(G(y)),
\label{cov-iv}
\end{align}
we construct $Y$ such that its support is in an interval
$[v^{*},v^{**}]\subseteq (0,1)$ and
\begin{itemize}
  \item for every $v\in [v^{*},v^{**}]$, we have
$\mathcal{C}_{\alpha}(v) >0$, thus assuring that the PQDE property holds,
provided that $Y$ takes only on values in the interval $[v^{*},v^{**}]$.
\end{itemize}
The construction of the aforementioned $Y$ is as follows.
We choose a set of $K$ points $v_k >v^{*}$
such that $\sum_{k=1}^{K-1} v_k <v^{**}$ and
$\sum_{k=1}^K v_k =1$. Define the cdf $G$ by the formula
\[
G(y)=\sum_{k=1}^{K} v_k \mathbf{1}\{ y_k\le y\},
\]
where $y_1<y_2<\cdots < y_K$ are real numbers. In other words,
the random variable $Y$ takes on
the values $y_k$ with the probabilities $v_k$.
Note that the range of the cdf $G$ is the set
$\{0,v_1, v_1+v_2, \dots, \sum_{k=1}^{K-1} v_k , 1\}$. By construction,
$\mathcal{C}_{\alpha}(v) >0$ for all
$v\in \{v_1, v_1+v_2, \dots, \sum_{k=1}^{K-1} v_k \}$. Hence,
in order to verify that $\mathcal{C}_{\alpha}(G(Y)) \ge 0$ for
all real $y\in \mathbf{R}$, we are only left to check that
$\mathcal{C}_{\alpha}(v) \ge 0$ for $v\in \{ 0,1\}$, but
this holds because $C_{\alpha}(u,0)=0$ and $C_{\alpha}(u,1)=u$.

In summary, we have constructed a pair $(U,Y)$ such that
$\mathbf{Cov}\big [U,\tau_y(Y)\big ]\ge 0$ for all $y\in \mathbf{R}$, that is,
$U (=X)$ is PQDE on $Y$, but the pair is not QD, that is,
it is neither PQD nor NQD.
This concludes Example \ref{example-21}.
\end{example}

\section{QDE-based Gr\"uss-type covariance bounds}
\label{qde-bounds}

From the previous two sections we know that the set of QDE random pairs
is larger than the set of QD pairs. In this sense, establishing
Gr\"uss-type covariance bounds under the QDE assumption would be
an extension of those established under the QD assumption. We explore
such QDE-based results in the current section. In what follows, we
use the notation
\[
\mathbf{A}_k[Z]=\big (\mathbf{E}\big [|Z-\mathbf{E}[Z] |^k\big ]\big )^{1/k}.
\]
The next theorem, whose proof is a consequence of
equation (\ref{cov-general-spec-2}), utilizes the QDE notion and
establishes a sharper bound than
Gr\"uss's original bound (\ref{cov-5e}).

\begin{theorem}\label{th-2.1}
For every pair $p,q\in (1,\infty )$ such that $p^{-1}+q^{-1}=1$,
we have the bound
\begin{equation}
\big |\mathbf{Cov}[X,\beta(Y)] \big |
\le \mathbf{D}_{p}[X,Y]\, \mathbf{G}_{p}[X,Y,\beta ],
\label{cov-5dd}
\end{equation}
where the QDE-based dependence coefficient is
\[
\mathbf{D}_{p}[X,Y]
=\sup_{y}{\big | \mathbf{Cov}\big [X,\tau_y(Y)\big ] \big |
\over \mathbf{A}_p[X] \mathbf{A}_q[\tau_y(Y)] }
\]
with the supremum taken over all $y\in \mathbf{R}$ such that $G(y)\in (0,1)$,
that is, over the support of the random variable $Y$,
and where the QDE-based Gr\"uss factor is
\[
\mathbf{G}_{p}[X,Y,\beta ]=\mathbf{A}_p[X]
\int \mathbf{A}_q[\tau_y(Y)]\, d|\beta|(y).
\]
\end{theorem}

Before discussing properties of the QDE-based dependence coefficient
and Gr\"uss's factor, we first show that bound (\ref{cov-5dd}) implies
Gr\"uss's bound (\ref{cov-5e}).

\begin{statement}\label{statem-5.1}
Setting $p=2$ and $\beta(x)=\beta_0(x)\equiv x$, we have that
under the Gr\"uss conditions on $X$ and $Y$,
Gr\"uss's bound (\ref{cov-5e}) follows from Theorem \ref{th-2.1}.
\end{statement}

\begin{proof}
Since $\mathbf{D}_{2}[X,Y]\le 1$,
we only need to show that
\begin{equation}
\mathbf{G}_{2}[X,Y,\beta_0 ]  \le {(A-a)(B-b)\over 4}.
\label{cov-gruss-1}
\end{equation}
Since $G(y)(1-G(y))$ does not exceed $1/4$ and is
equal to $0$ outside the support of $Y$, and since
$Y\in [b,B]\subset \mathbf{R}$ almost surely,
we have that $\mathbf{G}_{2}[X,Y,\beta_0 ]$
does not exceed $\sqrt{\mathbf{Var}[X]}\,(B-b)/2$.
Furthermore, since $X\in [a,A]\subset \mathbf{R}$ almost surely, then
(see, e.g., Zitikis \cite{Zitikis}, p.~16)
$\sqrt{\mathbf{Var}[X]}$ does not exceed $(A-a)/2$. Hence,
bound (\ref{cov-gruss-1}) holds.
\end{proof}

The dependence coefficient $\mathbf{D}_{p}[X,Y]$
is always in the interval $ [0,1]$.
It takes on the value $0$ when $X$ and $Y$ are independent.
Furthermore, the coefficient
achieves its upper bound $1$, as seen from the following statement.

\begin{statement}
The dependence coefficient $\mathbf{D}_{p}[X,Y]$ achieves its upper bound $1$.
\end{statement}

\begin{proof}
Given $Y$, let $X$ be the random variable $X_0$ defined by the equation
\[
X_0=\varepsilon
\big | \tau_{y_0}(Y)-\mathbf{E}[\tau_{y_0}(Y)]\big |^{q/p}
\mathrm{sign}\big ( \tau_{y_0}(Y)-\mathbf{E}[\tau_{y_0}(Y)] \big ) ,
\]
where
\begin{itemize}
\item
the number $y_0>0$ is any but fixed, and
\item
the random variable $\varepsilon $,
independent of $Y$, takes on the two values $\pm 1$
with same probabilities $p=1/2$.
\end{itemize}
The expectation $\mathbf{E}[X_0]$ is
equal to $0$. The absolute value $|X_0|$ is equal to
$\big | \tau_{y_0}(Y)-\mathbf{E}[\tau_{y_0}(Y)]\big |^{q/p}$.
The covariance $\mathbf{Cov}[X_0,\tau_y(Y)]$ is equal to
$\mathbf{A}_q^q[\tau_y(Y)]$. Hence, $\mathbf{D}_{p}[X_0,Y]=1$.
\end{proof}

The magnitude of the coefficient $\mathbf{D}_{p}[X,Y]$ depends on
the dependence structure between $X$ and $Y$,
as well as on the marginal cdf's of the two random variables.
For example, in the case
of independent $X$ and $Y$, we have $\mathbf{D}_{p}[X,Y]=0$.
Less trivial and thus more interesting examples follow.

\begin{example}\label{exa-1}\rm
Consider the convex combination of the lower and upper Fr\'echet copulas
as defined by equation (\ref{conv-1}). Both $X=U$ and $Y=V$ have
uniform distributions on the interval $[0,1]$, and thus
\begin{equation}
\mathbf{A}_p[U]={1\over 2(p+1)^{1/p}}
\quad \textrm{and} \quad
\mathbf{A}_q[\tau_v(V)]=\kappa_q^{1/q}(v) ,
\label{a-2}
\end{equation}
where the function $\kappa_q:[0,1]\to [0,1/2]$ is defined by
\[
\kappa_q(v)=v(1-v)^q+(1-v)v^q .
\]
Furthermore, since $\mathbf{Cov}\big [U,\tau_v(V)\big ]=\mathcal{C}_{\alpha}(v)$
with $\mathcal{C}_{\alpha}(v)$ given by equation (\ref{c-1}), we have that
\begin{equation}
\mathbf{D}_{p}[U,V]
= 2(p+1)^{1/p}\bigg | \alpha - {1\over 2} \bigg |
\sup_{0<v<1}{v(1-v)\over \kappa_q^{1/q}(v) } .
\label{d-1}
\end{equation}
Recall that $q=p/(p-1)$. The QDE-based Gr\"uss factor is
\begin{equation}
\mathbf{G}_{p}[U,V,\beta ]
= {1\over 2(p+1)^{1/p}}
\int_0^1 \kappa_q^{1/q}(v) d|\beta|(v).
\label{gd-1}
\end{equation}
In the special case when $p=2$ and $\beta(v)=\beta_0(v)\equiv v$,
we have that
\begin{align}
\mathbf{D}_{2}[U,V]
&=\sqrt{3}\, \bigg | \alpha - {1\over 2} \bigg |,
\label{h-1}
\\
\mathbf{G}_{2}[U,V,\beta_0 ]
&={1\over 2\sqrt{3}}\int_0^1 \sqrt{v(1-v)} \, dv
={\pi \over 16\sqrt{3}}.
\label{h-2}
\end{align}
Bound (\ref{cov-5dd}) therefore implies that
\begin{equation}
\big |\mathbf{Cov}[X,Y] \big |
\le {\pi \over 16}\, \bigg | \alpha - {1\over 2} \bigg |
\approx 0.19635 \, \bigg | \alpha - {1\over 2} \bigg | ,
\label{cov-5ddw}
\end{equation}
whereas the exact calculation using, for example, the equation
$\mathbf{Cov}[X,Y]=\int_0^1 \mathcal{C}_{\alpha}(v) dv $ and
formula (\ref{c-1}) gives the value
\begin{equation}
\mathbf{Cov}[X,Y]
= {1 \over 6}\, \bigg ( \alpha - {1\over 2} \bigg )
\approx 0.166667 \, \bigg ( \alpha - {1\over 2} \bigg ) .
\label{cov-5ddww}
\end{equation}
This concludes Example \ref{exa-1}.
\end{example}

\begin{example}\label{exa-2}\rm
Consider the convex combination of
the Farlie-Gumbel-Morgenstern and upper-Fr\'echet copulas
as defined by equation (\ref{conv-2}). Just like in the previous
example, both $X=U$ and $Y=V$ have
uniform distributions on the interval $[0,1]$. Thus,
equations (\ref{a-2}) hold in the current case as well.
The covariance $\mathbf{Cov}\big [U,\tau_v(V)\big ]$ is, however,
different: it is equal to $\mathcal{C}_{\alpha}(v)$
given by equation (\ref{c-2}).
Hence, the  QDE-based dependence coefficient is
\begin{equation}
\mathbf{D}_{p}[U,V]
= (p+1)^{1/p}\bigg | {4\over 3}\, \alpha -{1\over 3} \bigg |
\sup_{0<v<1}{v(1-v)\over \kappa_q^{1/q}(v) } .
\label{d-2}
\end{equation}
Recall that $q=p/(p-1)$.
Note that the QDE-based Gr\"uss factor $\mathbf{G}_{p}[U,V,\beta ]$ is
unaffected by the change of the dependence structure
and therefore has the same expression as in
previous Example \ref{exa-1} (see eq.~(\ref{gd-1})).
In the special case when $p=2$, from the above formulas we have that
\begin{equation}
\mathbf{D}_{2}[U,V]
={\sqrt{3}\over 2}\, \bigg | {4\over 3}\, \alpha -{1\over 3}  \bigg | .
\label{h-3}
\end{equation}
Bound (\ref{cov-5dd}) therefore implies that
\begin{equation}
\big |\mathbf{Cov}[X,Y] \big |
\le {\pi \over 32}\, \bigg | {4\over 3}\, \alpha -{1\over 3}  \bigg |
\approx  0.0981748\, \bigg | {4\over 3}\, \alpha -{1\over 3}  \bigg |,
\label{cov-5ddz}
\end{equation}
whereas the exact calculation using the equation
$\mathbf{Cov}[X,Y]=\int_0^1 \mathcal{C}_{\alpha}(v) dv $ and
formula (\ref{c-2}) gives the value
\begin{equation}
\mathbf{Cov}[X,Y]
= {1 \over 12}\, \bigg ( {4\over 3}\, \alpha -{1\over 3}  \bigg )
\approx  0.0833333\, \bigg ( {4\over 3}\, \alpha -{1\over 3}  \bigg ).
\label{cov-5ddzz}
\end{equation}
This concludes Example \ref{exa-2}.
\end{example}

\section{Estimating the QDE-based Gr\"uss factor}
\label{central-moments}

In Example \ref{exa-1} we calculated the QDE-based Gr\"uss factor
$\mathbf{G}_{p}[U,V,\beta ]$ in the case of uniform random
variables $U$ and $V$. In Statement \ref{statem-5.1} we estimated
$\mathbf{G}_{2}[X,Y,\beta_0 ]$ under the Gr\"uss condition on $X$ and $Y$.
In the current section we develop general results that aid in establishing
tight upper bounds for the QDE-based (general) Gr\"uss factor
$\mathbf{G}_{p}[X,Y,\beta]$. Specifically, upon expressing
the quantities $\mathbf{A}_p[X]$ and $\mathbf{A}_q[\tau_y(Y)]$
by the formulas
\[
\mathbf{A}_p[X]=\big ( \mathbf{E}\big [|X-\mu|^p\big ] \big )^{1/p}
\]
and
\[
\mathbf{A}_q[\tau_y(Y)]=\kappa_q^{1/q}(G(y)), 
\]
where the mean $\mu=\mathbf{E}[X]$, the cdf $G$ of $Y$, and the function
\[
\kappa_q(x)=x(1-x)^q+(1-x)x^q,
\]
we see that estimating $\mathbf{G}_{p}[X,Y,\beta]$ relies on
tight bounds for the $p^{\textrm{th}}$ central
moment $\mathbf{E}\big [|X-\mu|^p\big ]$ as well as on the function
$\kappa_q(x)$. Interestingly, as we shall see from Theorem \ref{th-main-1} below,
which is the main result of this section, tight upper bounds for
the $p^{\textrm{th}}$ central moment also
crucially rely on the function $\kappa_p(x)$.

\begin{theorem}\label{th-main-1}
Let $X$ be a random variable with support in $[a,A] $.
Then, for every $p\ge 1$, we have that
\begin{equation}
\mathbf{E}\big [|X-\mu|^p\big ]\leq (A-a)^{p}\kappa_p\bigg ( \frac{\mu-a}{A-a}\bigg ) .
\label{eq-1main}
\end{equation}
Consequently, with the notation $K_p=\sup_{x\in [0,1]} \kappa_p(x)$, we have that 
\begin{equation}
\mathbf{E}\big [|X-\mu|^p\big ]\leq (A-a)^{p}K_p .
\label{eq-2main}
\end{equation}
Furthermore, when $p\to \infty $, we have that
\begin{equation}
(1+p)K_{p}\to e^{-1}.
\label{eq-3main}
\end{equation}
The maximum $K_p$ of the function $\kappa_p(x)$
is achieved at a unique point $x=x_p$ in the interval $[1/(1+p),1/2]$
and thus, by symmetry, also at the point $1-x_p $ in the interval $[1/2,p/(1+p)]$.
The point $x_p$ is such that, when $p\to \infty $,
\begin{equation}
(1+p)x_{p}\to 1.
\label{eq-4main}
\end{equation}
\end{theorem}

\begin{note}\rm
Bound (\ref{eq-1main}) is sharp in the sense that there is a random
variable $X=X_1$ for which the inequality turns into an equality. Namely,
let $X_1$ take on only two values, $a$ and $A$, with the probabilities
$\mathbf{P}[X_1=a]=(A-\mu)/(A-a)$
and $\mathbf{P}[X_1=A]=(\mu-a)/(A-a)$, respectively.
Inequality (\ref{e-m}) becomes an equality.
\end{note}

\begin{note}\rm
When $p=2$, then $K_p=1/4$,
which plays a crucial role in deriving the Gr\"{u}ss bound.
Formulas for $K_p$ for the integers $1\le p \le 6$
are given in Table \ref{table-2}
\begin{table}[h!]
\centering
\[
\begin{array}{|r|r|r|}
\hline
{p} & {x_{p}} & {K_{p}} \\ \hline
{1} & {\frac{1}{2}}={0.50000000000000000000}                         & {\frac{1}{2}}={0.500000000000000000000} \\ \hline
{2} & {\frac{1}{2}}={0.50000000000000000000}                         & {\frac{1}{4}}={0.250000000000000000000} \\ \hline
{3} & {\frac{1}{2}}={0.50000000000000000000}                         & {\frac{1}{8}}={0.125000000000000000000} \\ \hline
{4} & {\frac{3-\sqrt{3}}{6}}={0.21132486540518711775}                & {\frac{1}{12}}={0.083333333333333333333} \\ \hline
{5} & {\frac{3-\sqrt{6\sqrt{10}-15}}{6}}={0.16776573020222127904}    & {\frac{5\sqrt{10}-14}{27}}={0.067088455586736913333} \\ \hline
{6} & {\frac{15-\sqrt{60\sqrt{10}-75}}{30}}={0.14294933504534875025} & {\frac{4\sqrt{10}-5}{135}}={0.056660078819803832059} \\ \hline
{7} & {0.12500637707104845945}                                       & {0.049087405277751670707} \\ \hline
{8} & {0.11111148199402853664}                                       & {0.043304947663997051030} \\ \hline
{9} & {0.10000001858448876931}                                       & {0.038742049800000743380} \\ \hline
{10} & {0.09090909172727279935}                                      & {0.035049389983188641270} \\ \hline
\end{array}
\]
\caption{The values of $x_{p}$ and $K_{p}$ for the integers $1\le p \le 10$.}
\label{table-2}
\end{table}
along with the corresponding values of $x_{p}\in [0,1/2]$.
By symmetry, the maximum $K_p$ is also achieved at the point
$1-x_{p} \in [1/2,1]$. However, throughout this paper
we use $x_{p}$ to denote the only existing point
in the interval $[0,1/2]$ such that
\[
K_{p}=\kappa_p(x_{p}).
\]
(We refer to the proof of Theorem \ref{th-main-1} for the existence
and uniqueness of $x_p$.)
When $p$ becomes large, expressions for $x_{p}$ and $K_p$ become unwieldy,
due to the fact that $x_{p}$ is a certain solution to
a polynomial equation of a high degree, for which explicit solutions
are not known to the best of our knowledge.
For this reason, in Table \ref{table-2} we have provided only
numerical values of $x_{p}$ and $K_p$ for the integers $7\le p \le 10$.
\end{note}

\begin{note}\rm
In the proof of Theorem \ref{th-main-1}
we shall establish lower and upper bounds for $x_{p}$ and $K_{p}$.
Namely, we shall show that
\begin{equation}\label{kp-shape-2m}
\frac{1}{p+1}\left( \frac{p}{1+p}\right) ^{p}\leq K_{p}\leq \frac{1}{p+1}
\left( \frac{p}{1+p}\right) ^{p}+\frac{1}{2^{1+p}}
\end{equation}
and
\begin{equation}\label{xp-vanish-46m}
{1\over p+1}\le x_p\quad
\left\{
  \begin{array}{ll}
    \displaystyle  ={1 \over 2}  & \hbox{ when }\, 1\le p \le 3,\\
    \displaystyle \le {1\over p+1}D_p   & \hbox{ when }\, p>3,
  \end{array}
\right.
\end{equation}
where
\begin{equation}\label{deltap}
D_p = 2\big / \bigg ( 1+\sqrt{p-3 \over p+1} ~\bigg ) .
\end{equation}
Note that $D_p \to 1$ when $p\to \infty $.
Bounds (\ref{xp-vanish-46m}) imply that
$x_{p}\sim 1/(p+1)$ when $p\to \infty $.
\end{note}

We next present a few results under additional assumptions on $X$.
For example, if we have more precise information
about the location of the mean $\mu $ than just $\mu \in [a,A]$
(see, e.g., Zitikis \cite{Zitikis} for a related discussion),
then the following corollary to Theorem \ref{th-main-1} holds.

\begin{corollary}
\label{prop-9}
Let $X$ be a random variable with
support in $[a,A] $, and let $[a^*,A^*]$ be a sub-interval of\,\, $[a,A]$
such that $\mu \in [a^*,A^*]$. Then for every $p\ge 1$ we have that
\begin{equation}
\mathbf{E}\big [|X-\mu|^p\big ]\leq (A-a)^{p}
\max\left\{\kappa_p(x)~:~
x\in \left[\frac{a^*-a}{A-a},\frac{A^*-a}{A-a}\right]\right\}.
\label{pr-1}
\end{equation}
\end{corollary}

\begin{proof}
Bound (\ref{pr-1}) follows from equation (\ref{equat-5})
and the bound (cf.\, bound (\ref{eq-1main}))
\[
\mathbf{E}\big [|Y-\mu ^{\prime }|^{p}\big ]
\le
\mu ^{\prime }(1-\mu ^{\prime })^{p}+\mu^{\prime }{}^{p}(1-\mu ^{\prime })
=\kappa_p(\mu ^{\prime }),
\]
where
\[
\mu ^{\prime }={\mu -a \over A-a }
\in \left[\frac{a^*-a}{A-a},\frac{A^*-a}{A-a}\right].
\]
This concludes the proof of Corollary \ref{prop-9}.
\end{proof}

In some situations the random variable $X$ might be
symmetric, in which case we have the following proposition.

\begin{proposition}
\label{prop-5}
Let $X$ be symmetric with support in $[a,A] $.
Then, for every $p\ge 1$,
\begin{equation}
\mathbf{E}\big [|X-\mu|^p\big ]\leq  \frac{(A-a)^{p}}{2^{p}}.
\label{equat-5j}
\end{equation}
\end{proposition}

\begin{proof}
Denote $Y=(X-a)/(A-a)$. The random variable $Y$ has support
in $[0,1]$ and its mean is $\mu ^{\prime }=(\mu -a)/(A-a)$. Hence,
\begin{equation}
\frac{\mathbf{E}\big [|X-\mu|^p\big ]}{(A-a)^{p}}
=\mathbf{E}\big [|Y-\mu ^{\prime }|^{p}\big ].
\label{equat-5}
\end{equation}
Since $Y$ is symmetric (around its mean $\mu ^{\prime }$),
and the mean $\mu ^{\prime }$ is in the interval $[0,1]$, we have that
$|Y-\mu ^{\prime }|$ does not exceed $1/2$. Hence, the right-hand
side of equation (\ref{equat-5}) does not exceed $1/2^p$.
This concludes the proof of Proposition \ref{prop-5}.
\end{proof}

\begin{proof}[Proof of Theorem \ref{th-main-1}]
Edmundson \cite{Edmundson} proved that if
$f$ is a convex function and $X$ is a random variable with support
in $[a,A] $, then
\begin{equation}\label{e-m}
\mathbf{E}[f(X)]\leq\frac{A-\mu}{A-a}f(a)+\frac{\mu-a}{A-a}f(A),
\end{equation}
where $\mu$ is the mean of $X$.
This result was subsequently extended by Madansky \cite{Madansky} and is
nowadays known as the Edmundson-Madansky inequality.
Since the function $f(x)=|x-\mu|^p$ is convex for every $p\ge 1$,
the Edmundson-Madansky inequality gives
bound (\ref{eq-1main}). Bound (\ref{eq-2main}) follows trivially.
Statement (\ref{eq-3main})  follows from the fact that
$(p/(1+p))^{p}\to e^{-1}$ and the bound
\begin{equation}\label{kp-shape-2}
\frac{1}{p+1}\left( \frac{p}{1+p}\right) ^{p}\leq K_{p}\leq \frac{1}{p+1}
\left( \frac{p}{1+p}\right) ^{p}+\frac{1}{2^{1+p}},
\end{equation}
which holds for every $p\ge 1$ as we shall next prove.
We start with the upper bound and show that, for every $x\in \lbrack 0,1]$,
\begin{equation}\label{kp-shape-3}
\kappa_p(x)\leq \frac{1}{p+1}\left( \frac{p}{1+p}\right) ^{p}+\frac{1}{2^{1+p}}.
\end{equation}
Since $\kappa_p(x)=\kappa_p(1-x)$, we only need to check
bound (\ref{kp-shape-3}) for $x\in [0,1/2]$.
With the notation $h(x)=x(1-x)^{p}$ we have that
\[
\kappa_p(x)=h(x)+h(1-x).
\]
On the interval $[0,1/2]$, the function $h(x)$ achieves its maximum
at the point $x=1/(p+1)$, and so we have that
\begin{equation}\label{hhh-1}
h(x)\leq \frac{1}{p+1}\left( \frac{p}{1+p}\right) ^{p}.
\end{equation}
On the other hand, the function $h(1-x)$ is increasing on
the interval $[0,1/2]$, and so it achieves its maximum at the point $x=1/2$,
thus giving the bound
\begin{equation}\label{hhh-2}
h(1-x)\leq \dfrac{1}{2^{p+1}} \, .
\end{equation}
Adding up bounds (\ref{hhh-1}) and (\ref{hhh-2}),
we obtain bound (\ref{kp-shape-3}).

To establish the lower bound of (\ref{kp-shape-2}),
we first note that $\kappa_p(x)\ge h(x)$, and thus $K_p\ge h(x)$.
Since the function $h(x)$ achieves its maximum on the interval $[0,1/2]$
at the point $x=1/(p+1)$, bound $K_p\ge h(x)$ implies
the lower bound of (\ref{kp-shape-3}). This completes the proof
of the two bounds of (\ref{kp-shape-2}).

To prove that $(1+p)x_{p}\to 1$ when $p\to \infty $, we show that
\begin{equation}\label{xp-vanish-47}
{\frac{1}{p+1}} \le x_p\le {1\over p+1} D_p ,
\end{equation}
where $D_p$ is defined by equation (\ref{deltap}).
The lower bound of (\ref{xp-vanish-47}) is obvious when $p=1$.
Hence, from now on we consider the case $p>1$. We want to find
those $x\in [0,1/2]$ that maximize the function $\kappa_p(x)$.
We check that $\kappa_p'(x)=0$ holds if and only if
\begin{equation}\label{xp-vanish-3}
\bigg ( {x\over 1-x}\bigg )^{p-1}= { (1+p)x-1 \over p-(1+p)x}.
\end{equation}
The left-hand side of equation (\ref{xp-vanish-3}) is non-negative
for all $x\in [0,1/2]$, whereas the right-hand side is non-negative
only when $x\ge 1/(p+1)$.
This implies that if we find a point $x\in [0,1/2]$ such that
the function $\kappa_p(x)$ is maximized, then the point must be
such that $x\ge 1/(p+1)$, thus implying
the left-hand bound of (\ref{xp-vanish-47}).

Both sides of equation (\ref{xp-vanish-3}) are increasing functions
on the interval  $[1/(p+1),1/2]$. Both sides are equal to $1$ at
the end point $x=1/2$. Hence, the existence and
uniqueness of $x_p\in [1/(p+1),1/2)$ is equivalent
to showing that the two functions intersect only once
on the interval $[1/(p+1),1/2)$, but if they do not intersect, then
we have $x_p=1/2$. We determine this by checking if the ratio
\begin{equation}\label{xp-vanish-45}
R(x)={ (1+p)x-1 \over p-(1+p)x}
\Big /
\bigg ( {x\over 1-x}\bigg )^{p-1}
\end{equation}
crosses the horizontal line $\{(x,1): x\in [1/(p+1),1/2]\}$
only once, provided that it crosses at all. Note that $R(1/(p+1))=0$
and $R(1/2)=1$. Furthermore, the derivative $R'(x)$ 
is always positive on the interval $[1/(p+1),1/2)$ when
$1\le p \le 3$. Hence, the function $R(x)$ achieves its maximum at the end
point $x=1/2$, implying that $x_p=1/2$ when $1\le p \le 3$.
When $p>3$, then the derivative $R'(x)$ is
positive on $[1/(p+1),x_p^*)$, negative on $(x_p^*,1/2)$,
and vanishes at $x=x_p^*$, where
\begin{equation}\label{xp-vanish-46}
x_p^*={1\over 2} \bigg ( 1-\sqrt{p-3 \over p+1} \bigg ).
\end{equation}
Since $R(1/(p+1))=0$ and $R(1/2)=1$, we therefore
conclude that when $p> 3$, then the function $R(x)$ is increasing on the interval
$[1/(p+1),x_p^*]$ with the initial value $R(1/(p+1))=0$,
and then, once it reaches its maximum at the point $x_p^*$,
becomes decreasing on the interval $[x_p^*,1/2]$
with the final value $R(1/2)=1$. Since the final
value is $1$, we conclude that the function $R(x)$ has
crossed the horizontal line $\{(x,1): x\in [1/(p+1),1/2]\}$
exactly once. The crossing point is $x_p$ because it maximizes
the function $\kappa_p(x)$.
This proves Theorem \ref{th-main-1}.
\end{proof}

\section{Regression-based covariance bounds}
\label{regression}

In this section, we look at the covariance $\mathbf{Cov}[X,\beta(Y)] $
from a slightly different angle. First, we write the equation
\begin{equation}
\mathbf{Cov}[X,\beta(Y)]
=\mathbf{Cov}[\alpha(Y),\beta(Y)],
\label{cov-5vi}
\end{equation}
where
\[
\alpha(y)=r_{X|Y}(y)\equiv \mathbf{E}[X|Y=y]-\mathbf{E}[X],
\]
which is the centered regression function.
The right-hand side of equation (\ref{cov-5vi}) complicates
the problem by introducing an additional distortion function
but it also simplifies the problem by reducing the pair $(X,Y)$ to $(Y,Y)$.
Nevertheless, the following theorem, whose proof is based on
an application of H\"older's inequality on the right-hand
side of equation (\ref{cov-5vi}), offers a sharper bound than
Gr\"uss's bound (\ref{cov-5e})
by utilizing a regression-based dependence coefficient
\[
\mathbf{\Delta}_p[X,Y]={\mathbf{A}_p[r_{X|Y}(Y)] \over \mathbf{A}_p[X] } .
\]
We shall discuss properties of the coefficient later in this section.

\begin{theorem}\label{th-2.2}
For every pair $p,q\in (1,\infty )$ such that $p^{-1}+q^{-1}=1$,
we have the bound
\begin{equation}
\big |\mathbf{Cov}[X,\beta(Y)] \big |
\le \mathbf{\Delta}_p[X,Y]\, \mathbf{\Gamma}_{p}[X,Y,\beta ],
\label{cov-5hh}
\end{equation}
where
\begin{equation}
\mathbf{\Gamma}_{p}[X,Y,\beta ]=\mathbf{A}_p[X]\mathbf{A}_q[\beta(Y)]
\label{cov-5gg}
\end{equation}
is the regression-based Gr\"uss factor of the bound.
\end{theorem}

We next show that bound (\ref{cov-5hh}) implies
Gr\"uss's bound (\ref{cov-5e}).

\begin{statement}
Setting $p=2$ and $\beta(x)=\beta_0(x)\equiv x$, we have that
under the Gr\"uss conditions on $X$ and $Y$,
Gr\"uss's bound (\ref{cov-5e}) follows from Theorem \ref{th-2.2}.
\end{statement}

\begin{proof}
Since $\mathbf{\Delta}_{2}[X,Y]\le 1$,
we only need to show that
\begin{equation}
\mathbf{\Gamma}_{2}[X,Y,\beta_0 ]  \le {(A-a)(B-b)\over 4}.
\label{cov-gruss-1o}
\end{equation}
Since $X\in [a,A]\subset \mathbf{R}$ almost surely, then
(see, e.g., Zitikis \cite{Zitikis}, p.~16)
$\mathbf{A}_2[X]=\sqrt{\mathbf{Var}[X]}$
does not exceed $(A-a)/2$. Likewise,
when $Y\in [b,B]\subset \mathbf{R}$ almost surely, then
$\mathbf{A}_2[\beta(Y)]=\sqrt{\mathbf{Var}[Y]}$
does not exceed $(B-b)/2$. Bound (\ref{cov-gruss-1o}) follows.
\end{proof}

We next discuss properties of the regression-based dependence coefficient
$\mathbf{\Delta}_p[X,Y]$. Note first that the coefficient is
always in the interval $ [0,1]$. Furthermore,
when $X$ and $Y$ are independent, then $\mathbf{\Delta}_p[X,Y]=0$, and
when $X=Y$ almost surely, then $\mathbf{\Delta}_p[X,Y]=1$.

When the pair $(X,Y)$ follows the bivariate normal distribution,
then the centered regression function takes on the form
\begin{equation}
r_{X|Y}(y)={\mathbf{Cov}[X,Y]\over \mathbf{Var}[Y]}\big (y-\mathbf{E}[Y] \big ),
\label{cov-regr-1}
\end{equation}
and we therefore have the equation
\[
\mathbf{\Delta}_p[X,Y]={\mathbf{Cov}[X,Y]\over \mathbf{Var}[Y]}
{\mathbf{A}_p[Y] \over \mathbf{A}_p[X] } .
\]
In particular, when $p=2$, since $\mathbf{A}_2[X]=\sqrt{\mathbf{Var}[X]}$ and
$\mathbf{A}_2[\beta(Y)]=\sqrt{\mathbf{Var}[Y]}$,
we have that the regression-based dependence coefficient
$\mathbf{D}_2[X,Y]$ is equal to
the Pearson correlation coefficient $\mathbf{Corr}[X,Y]$.
Furthermore, an application of equation (\ref{cov-regr-1})
on the right-hand side of equation (\ref{cov-5vi})
gives equation (\ref{cov-10i}), which, assuming that
$\beta $ is differentiable, in turn gives equation (\ref{cov-10stein}).

\section*{Acknowledgements}

The four authors have been supported by
the research grant FRG1/10-11/012 from Hong Kong Baptist University (HKBU)
under the title ``The Covariance Sign of Transformed Random Variables
with Applications to Economics, Finance and Insurance''.
The authors also gratefully acknowledge their partial research support by
the Agencia Nacional de Investigaci\'{o}n e Innovaci\'{o}n (ANII) of Uruguay,
the Research Grants Council (RGC) of Hong Kong,
and the Natural Sciences
and Engineering Research Council (NSERC) of Canada.


\begin{thebibliography}{99}


\bibitem{BrollEtAl-IMAJMM}
Broll, U., Egozcue, M., Wong, W.K.\, and Zitikis, R.,
Prospect theory, indifference curves, and hedging risks.
Applied Mathematics Research Express 2010 (2010) 142--153.

\bibitem{Cuadras}
Cuadras, C.M.,
{On the covariance between functions}.
Journal of Multivariate Analysis {81} (2002) 19--27.

\bibitem{Denuit}
Denuit, M., Dhaene, J., Goovaerts, M. and Kaas, R.,
Actuarial Theory for Dependent Risks: Measures, Orders and Models.
Wiley, Chichester, 2005.


\bibitem{Dragomir}
Dragomir, S.S.,
A generalization of Gr\"uss's inequality in
inner product spaces and applications.
Journal of Mathematical Analysis and Applications
{237} (1999) 74--82.

\bibitem{DragomirAgarwal}
Dragomir, S.S. and Agarwal, R.P.,
Some inequalities and their application for estimating the
moments of guessing mappings.
Mathematical and Computer Modelling {34} (2001) 441--468.

\bibitem{DragomirDiamond}
Dragomir, S.S. and Diamond, N.T.,
A discrete Gr\"{u}ss type inequality and applications for the
moments of random variables and guessing mappings.
In: Stochastic Analysis and Applications.
Vol. 3, Nova Science Publishers, New York, 2003, pp.~21--35.

\bibitem{Dragomir-book}
Dragomir, S.S.,
Advances in Inequalities of the Schwarz, Grüss and Bessel
Type in Inner Product Spaces.
Nova Science, New York, 2005.

\bibitem{EgozcueEtAl-2009}
Egozcue, M., Fuentes Garcia, L., and Wong, W.K.,
On some covariance
inequalities for monotonic and non-monotonic functions.
Journal of Inequalities in Pure and Applied Mathematics
{10}, Article 75 (2009) 7 pages.

\bibitem{EgozcueEtAl-JIA}
Egozcue, M., Fuentes Garc\'ia, L., Wong, W.K.\, and Zitikis, R.,
Gr\"{u}ss-type bounds for the covariance of transformed random variables.
Journal of Inequalities and Applications
2010 (2010) Article ID 619423, 10 pages.


\bibitem{EgozcueEtAl-IMAJMM}
Egozcue, M., Fuentes Garc\'ia, L., Wong, W.K.\, and Zitikis, R.,
The covariance sign of transformed random variables with applications to economics and finance.
IMA Journal of Management Mathematics (2010), to appear. Available on-line at:\\
http://imaman.oxfordjournals.org/content/early/2010/09/07/imaman.dpq012

\bibitem{Edmundson}
Edmundson, H.P.,
Bounds on the expectation of a convex
function of a random variable. Technical report, The RAND Corporation,
Paper P-982, Santa Monica, California, 1957.


\bibitem{furman-zitikis-jipam}
Furman, E.\, and Zitikis, R.,
Monotonicity of ratios involving incomplete gamma
functions with actuarial applications.
Journal of Inequalities in Pure and Applied Mathematics
9 (2008), Article 61, 6 pages.

\bibitem{furman-zitikis-jmaa}
Furman, E.\, and Zitikis, R.,
A monotonicity property of the composition
of regularized and inverted-regularized gamma functions
with applications.
Journal of Mathematical Analysis and Applications
{348} (2008) 971--976.

\bibitem{furman-zitikis-IME-1}
Furman, E.\, and Zitikis, R.,
Weighted premium calculation principles.
Insurance: Mathematics and Economics 42 (2008) 459--465.

\bibitem{furman-zitikis-IME-2}
Furman, E.\, and Zitikis, R.,
Weighted risk capital allocations.
Insurance: Mathematics and Economics 43 (2008) 263--269.


\bibitem{furman-zitikis-naaj}
Furman, E.\, and Zitikis, R.,
Weighted pricing functionals with applications
to insurance: an overview.
North American Actuarial Journal
{13} (2009) 1--14.

\bibitem{furman-zitikis-astin}
Furman, E.\, and Zitikis, R.,
General Stein-type covariance decompositions with
applications to insurance and finance.
ASTIN Bulletin 40 (2010) 369--375.

\bibitem{gf-2007}
Genest, C. and Favre, A.-C.,
Everything you always wanted to know about copula modeling but were afraid to ask.
Journal of Hydrologic Engineering 12 (2007) 347--368.

\bibitem{ggb-2009}
Genest, C., Gendron, M. and Bourdeau-Brien, M.,
The advent of copulas in finance.
European Journal of Finance 15 (2009) 609--618.

\bibitem{Genest-Ghoudi}
Genest, C. and Ghoudi, K.,
Une famille de lois bidimensionnelles insolite.
Comptes Rendus de l'Acad\'{e}mie des Sciences de Paris
318, Series I (1994) 351--354.

\bibitem{Genest--MacKay-CJS}
Genest, C. and MacKay, R.J.,
Copules archim\'{e}diennes et familles de lois
bidimensionnelles dont les marges sont donn\'{e}es.
Canadian Journal of Statistics 14 (1986) 145--159.

\bibitem{Genest--MacKay-AS}
Genest, C. and MacKay, R.J.,
The joy of copulas: bivariate distributions with uniform marginals.
American Statistician 40 (1986) 280--283.



\bibitem{Gruss}
Gr\"uss, G.,
{\"Uber das maximum des absoluten betrages von} $%
\frac 1 {b-a} \int_{a}^{b} f(x) g(x) dx - \frac 1 { ( b-a)^2}
\int_{a}^{b} f(x) dx \int_{a}^{b} g(x) dx$.
Mathematische Zeitschrift {39} (1935) 215--226.

\bibitem{Hoeffding}
Hoeffding, W.,
{Masstabinvariante korrelationstheorie},
Schriften des Matematischen Instituts f\"{u}r Angewandte Matematik der
Universit\"{a}t Berlin {5} (1940) 181--233.

\bibitem{IzuminoPecaric}
Izumino, S. and Pe\v{c}ari\'c, J.E.,
Some extensions of Gruss' inequality and its applications.
Nihonkai Mathematical Journal {13} (2002) 159--166.

\bibitem{IzuminoPecaricTepes}
Izumino, S., Pe\v{c}ari\'c, J.E. and Tepe\v{s}, B.,
A Gr\"{u}ss-type inequality and its applications.
Journal of Inequalities and Applications  {2005} (2005) 277--288.


\bibitem{Kowalczyk-Pleszczynska}
Kowalczyk, T. and Pleszczynska, E.,
Monotonic dependence functions of bivariate distributions.
Annals of Statistics {5} (1977) 1221-1227.

\bibitem{Lehmann}
Lehmann, E.L.,
Some concepts of dependence.
Annals of Mathematical Statistics {37} (1966) 1137--1153.

\bibitem{McNeil-Frey-Embrechts}
McNeil, A.J., Frey, R. and Embrechts, P.,
Quantitative Risk Management: Concepts, Techniques, and Tools.
Princeton University Press, Princeton, 2005.

\bibitem{Madansky}
Madansky, A.,
Bounds on the expectation of a convex function
of a multivariate random variable.
Annals of Mathematical Statistics 30 (1959) 743--746.


\bibitem{Mitrinovi}
Mitrinovi\'c, D., Pe\v cari\'c, J.E. and Fink, A.M.,
Classical and New Inequalities in Analysis, Kluwer, Dordrecht, 1993.


\bibitem{nelsen}
Nelsen, R.B.,
An Introduction to Copulas.
(Second Edition.) Springer, New York, 2006.

\bibitem{Sen}
Sen, P.K.,
The impact of Wassily Hoeffding's research on
nonparametric. In: Collected Works of Wassily Hoeffding, N.I. Fisher
and P.K. Sen (eds.), Springer, New York, 1994, pp.~29--55.

\bibitem{SeondovEtAl-new}
Sendov, H.S., Wang, Y.\, and Zitikis, R.,
Log-supermodularity of weight functions and the loading
monotonicity of weighted insurance premiums.
Insurance: Mathematics and Economics (2010), under revision.
Available at SSRN:
http://ssrn.com/abstract=1660809; and
also at arXiv: http://arxiv.org/abs/1008.3427


\bibitem{Steele}
Steele, J.M.,
The Cauchy-Schwarz Master Class: An Introduction to
the Art of Mathematical Inequalities.
Cambridge University Press, Cambridge, 2004.

\bibitem{Stein}
Stein, C.M.,
Estimation of the mean of a multivariate normal distribution.
In: Proceedings of Prague Symposium on Asymptotic Statistics (ed. J. Hajek),
pp. 345--381, Prague, Charles University, 1973.


\bibitem{wright}
Wright, R.,
Expectation dependence of random variables, with an application
in portfolio theory.
Theory and Decision {22} (1987) 111--124.

\bibitem{Zitikis}
Zitikis, R.,
Gr\"{u}ss's inequality, its probabilistics
interpretation, and a sharper bound.
Journal of Mathematical Inequalities {3} (2009) 15--20.
\end{thebibliography}
\end{document}